\documentclass[12pt]{paper}
 \usepackage{geometry}
\usepackage{hyperref}
\usepackage{float}
                \usepackage{amsthm,amsfonts,amsmath,amscd,amssymb,epsfig,verbatim,mathabx
}

\usepackage{graphicx}
\usepackage{epstopdf}
 {\title{The space of tight contact structures on $\R^3$ is contractible} 
\author{Yakov Eliashberg\thanks{Partially supported by the NSF grants   DMS-1807270 and DMS-2104473} \\ Stanford University\\ USA  \and   Nikolai Mishachev 
\\Lipetsk Technical University\\Russia} 
\date{}  
 \setlength{\marginparwidth}{1.2in}
\let\oldmarginpar\marginpar
\renewcommand\marginpar[1]{\-\oldmarginpar[\raggedleft\footnotesize #1]%
{\raggedright\footnotesize #1}}

\parindent=0pt
\parskip=4pt
\hyphenation{ma-ni-fold ma-ni-folds sub-ma-ni-fold sub-ma-ni-folds}
\theoremstyle{plain}
\newtheorem{theorem}{Theorem}[section]
\newtheorem{thm}[theorem]{Theorem}

\newtheorem{cor}[theorem]{Corollary}

\newtheorem{prop}[theorem]{Proposition}
\newtheorem{lemma}[theorem]{Lemma}

\theoremstyle{remark}

\newtheorem{remark}[theorem]{Remark}
\newtheorem*{remark*}{Remark}
\theoremstyle{Example}
\newtheorem{Example}[theorem]{Example}

\theoremstyle{definition}
 
\newcommand{\wt}{\widetilde}
\newcommand{\wh}{\widehat}

\newcommand{\ol}{\overline}

\newcommand{\p}{\partial}

\newcommand{\eps}{\varepsilon}

\newcommand{\Tight}{\mathrm {Tight}}

\newcommand{\R}{{\mathbb{R}}}
\newcommand{\C}{{\mathbb{C}}}

\newcommand{\e}{{\bf e}}

\newcommand{\Int}{{\rm Int\,}} 

\renewcommand{\min}{{\rm min}}
\renewcommand{\max}{{\rm max}}

\newcommand{\loc}{{\rm loc}}

\newcommand{\Id}{\mathrm {Id}}

\newcommand{\Ker}{\mathrm{Ker\,}}
\newcommand{\Coker}{\mathrm{Coker\,}}

\newcommand{\Diff}{\mathrm{Diff}}

\newcommand{\Span}{\mathrm{Span}}

\newcommand{\LL}{\mathcal{L}}
\newcommand{\FF}{\mathcal{F}}

\newcommand{\PP}{\mathcal{P}}

\newcommand{\TT}{\mathcal{T}}
\newcommand{\HH}{\mathcal{H}}

\newcommand{\ZZ}{\mathcal{Z}}

\def\Op{{\mathcal O}{\it p}\,}
\numberwithin{figure}{section}
\renewcommand{\Vert}{\mathrm {Vert}}
\newcommand{\Uert}{\mathrm {Uert}}
\newcommand{\Tr}{\mathrm {Tr}}
\newcommand{\Simple}{\mathrm {Simple}}
\begin{document}

\maketitle
\rightline{\sl To Emmanuel Giroux on his 60th anniversary}
\begin{abstract}

 One of the results  of the   paper \cite{El92} was the proof that any tight contact structure on $S^3$ is diffeomorphic to the standard one. It was also claimed there without a proof that   similar methods could be used to  prove  a multi-parametric version: the space of tight contact structures on $S^3$, fixed at a point, is contractible.
We prove this result  in the current paper.

 \end{abstract}

 \tableofcontents
\section{Introduction}\label{sec:intro}

A contact structure $\xi$ is called {\em overtwisted}, see \cite{El89}, if there exists an embedded $2$-disc $D\subset \R^3$ which is tangent to $\xi$ along $\p D$. A non-overtwisted contact structure is called {\em tight}.
It is a fundamental result of D.~Bennequin, see \cite{Be83}, that  the standard contact structure $\zeta_0=\{dz-ydx=0\} $ on $\R^3$  is tight.

The following theorem is the main result of the current paper.
\begin{thm}\label{thm:main}
The space of standard at infinity tight contact structures on $\R^3$ is contractible.
\end{thm}

A contact structure defines an orientation of a contact 3-manifold. Given an oriented contact manifold we call a contact structure {\em positive} if the   contact orientation is   the given one.
\begin{cor}\label{cor:main}
The space $\Tight_+(S^3)$ of positive  tight contact structures on $S^3$ is homotopy equivalent to $\R {\mathrm P}^2$.
\end{cor}
Indeed, an evaluation map $\Tight_+(S^3)\to\R  {\mathrm P}^2$, associating with a contact structure a non-oriented contact plane at a fixed point, is a Serre fibration with a fiber homotopy equivalent to the space of standard at infinity tight contact structures on $\R^3$.
\medskip
 
The non-parametric version of Theorem \ref{thm:main}, i.e. that {\em the space of standard at infinity tight contact structures on $\R^3$ is connected} was proven in \cite{El92}. Equivalently, that result means that any standard at infinity tight contact structure on $\R^3$ is diffeomorphic to $\zeta_0$ via a compactly supported   diffeomorphism of $\R^3$.  
 An approach to the proof of   the parametric case  using convex surface
theory has been suggested and partially implemented by D.~J\"{a}nichen, see \cite{Ja18}.  The proof presented in the current paper does not use the theory of contact convexity.

Denote by $\Diff_0(\R^3)$ the group of compactly supported diffeomorphisms of $\R^3$, and by
$\Diff_0(\R^3,\zeta_0)$ the group of compactly supported contactomorphisms of $(\R^3,\zeta_0)$.
The non-parametric  version  Theorem  \ref{thm:main} from \cite{El92}  together with Gray's theorem \cite{Gra59} implies that  the evaluation map $f\mapsto f_*\zeta_0$ is  a  a Serre fibration 
$\Diff_0(\R^3)\to\Tight_0(\R^3)$, where 
$\Tight_0(\R^3)$ is the space of standard at infinity tight contact structures on $\R^3$.      The fiber of this fibration is 
$\Diff_0(\R^3,\zeta_0)$. Hence,   Theorem  \ref{thm:main}    equivalently means that   {\em the inclusion map
$j:\Diff_0(\R^3,\zeta_0))\to \Diff_0(\R^3)$ is a homotopy equivalence}, which in view of  A. Hatcher's theorem \cite{Ha83}
implies that the group  $\Diff_0(\R^3,\zeta_0)$ is contractible.

 A recent paper \cite{FMAP20} by E. Fern\'{a}ndez, J. Martinez-Aguinaga and  F. Presas   used the main result of the current paper for the study of the topology of the  group of contactomorphisms of various  other $3$-manifolds.

\subsubsection*{Scheme  of the proof and the plan of the paper}
As in \cite{El92}, the proof is based on the analysis of characteristic foliations on the 2-sphere induced by a family of  tight contact structures on its  neighborhood.
To make the topology of a characteristic foliation manageable we   arrange  that its singularities   are of Morse or generalized Morse type. This is achieved  in Proposition \ref{prop:Igusa}, which is  an analog for our situation of Igusa's theorem about functions with moderate singularities,  see \cite{Ig84} and \cite{EM00, EM12}.

The second   ingredient in the proof of the main result is  a new characterization of characteristic foliations  induced on a sphere by a tight contact structure, see Proposition \ref{prop:tight-2-sphere}. It is formulated   in terms of existence of a Lyapunov function with special properties.  We call Lyapunov   functions in this class {\em simple taming functions}. Let  us recall that    Giroux's  criterion  from \cite{Gir91} for tightness of characteristic foliations is applicable only to convex (in a contact sense) surfaces, i.e. surfaces admitting a transverse contact vector field. Our  characterization is  applicable to any generalized Morse foliation  on a 2-sphere without  the    contact convexity assumption. The proof of Proposition \ref{prop:tight-2-sphere} is based on the analysis of the topology of tight characteristic foliations in Section \ref{sec:allowable}, see there Proposition \ref{prop:key}.
This leads to a Proposition \ref{prop:taming-parametric} which allows us to construct a family of taming simple functions for any family of tight foliations on a 2-sphere.

The third ingredient in the proof is  Proposition \ref{prop:simple-extension}, which provides homotopically canonical extension of a family of simple functions  to a family of functions  on the ball without critical points and with contractible components of its level sets. This proposition is purely topological and has nothing to do with contact geometry.

The final ingredient is Proposition \ref{prop:ext-fol-cont} which uses the extended to the ball taming functions for construction of homotopically canonical extension of a  family of  characteristic foliations on the sphere  to   a family of contact structures on the ball. We apply for this purpose complex geometric techniques of strictly pseudoconvex hypersurfaces, though,     probably,  it could be achieved  by more direct contact geometric methods.
  We conclude the proof of Theorem \ref{thm:main}  in Section \ref{sec:proof-main}.
 
\medskip\noindent  {\bf Acknowledgements.} We are grateful to M. Jitomirskii for enlightening discussions about singularities of planar vector fields, and to Y.S. Ilyashenko for  providing a reference to the paper \cite{IY91}.  

 \section{Characteristic foliation on  a surface  in a tight contact manifold}\label{sec:char-fol}
 \subsection{Singularities of  a characteristic foliation}
Given a contact structure $\xi=\{\alpha=0\}$ on a neighborhood of a $2$-sphere $S $ in a $3$-manifold, we use the term {\em characteristic foliation} for the singular line field defined by the Pfaffian equation $\{\alpha|_{S}=0\}$, as well as for  the singular foliation
  $\FF:=\{\alpha|_{S}=0\}$  to which it integrates. A characteristic foliation is called {\em Morse} if it has no limit cycles and  all its singular points  are non-degenerate. 

More precisely, in a neighborhood of a singular point $p\in S$ we have $d\alpha|_S\neq 0$, and hence $\alpha|_S$ is a Liouville form   for the symplectic form $d\alpha|_S$. The corresponding Liouville field $Z$, $\iota(Z)(d\alpha|_S)=\alpha|_S$, integrates to   the characteristic foliation $\FF$.   In any local  coordinate system $u=(x,y)$   centered at $p$  the vector field $Z$ is given by a differential equation $\dot u= f(u)$, $f(0)=0$.  We are not assuming the coordinate system canonical (i.e. that $d\alpha=dx\wedge dy$), but require  that it defines the symplectic orientation. The linear part $A=d_0f$ has $\Tr A> 0$.
The singular point $p$ is called {\em non-degenerate} if $A$ is non-degenerate.  A  non-degenerate point   is  called {\em elliptic} if $\det A>0$ and {\em hyperbolic} otherwise. In the hyperbolic case $A$ has  two real  eigenvalues $\lambda_1>0$ and $\lambda_2<0$. In the elliptic case eigenvalues  are either positive real numbers, or conjugate complex numbers with the real part equal to $\frac{\Tr A}2$.  
  
A singular point is called an {\em embryo}  point    if $A$ has rank $1$ and if the    second differential $d^2f:\Ker A\to \Coker A$  does not vanish.   We note   that  the linear map $d^2f:\Ker A\to \Coker A$ is invariantly defined up to a (non-zero) scalar factor.

We call  a characteristic  foliation {\em generalized Morse} if all its singularities are either nondegenerate or embryos.  SIngularities of a generalized Morse foliation are shown on Fig. \ref{fig:ell-hyp-embryo}.   

 It follows from the results of F. Takens, see \cite{Ta74}, that  in a neighborhood of an embryo the directing Liouville field  $Z$ is   orbitally equivalent  to (i.e. diffeomorphic  to a field proportional to)  the field  $ x \frac{\p}{\p x}+ y^2f(y)\frac{\p}{\p y} ,\; f(0)\neq 0.$   \footnote{ Moreover, one can choose coordinates in such a way that $f(y)$ is  equal either to $1$, or $1+y$, but we will not need this stronger statement.}
 The above normalization claim also holds in a parametric form.

 \begin{lemma}\label{lm:Takens-param}
 Let $\Lambda$ be a compact parameter space and $Z_\lambda,\lambda\in\Lambda,$  a family of  $C^\infty$-vector fields on  a  neighborhood $\Op_{\R^2}0\subset\R^2$ 
  of  the origin in $\R^2$  with an embryo singularity at the origin.
 Then for any $k>0$ there exists a family of germs of $C^k$-diffeomorphisms $h_\lambda:(\R^2,0)\to(\R^2,0)$ such that $$(h_\lambda)_*Z_\lambda=g_\lambda(x,y)\left(x \frac{\p}{\p x}+ y^2f_\lambda(y)\frac{\p}{\p y}\right) ,\; f_\lambda(0),g_\lambda(0,0)\neq 0.$$

 \end{lemma}
 
    \begin{figure}[h]
\centerline{\includegraphics[width=11cm]{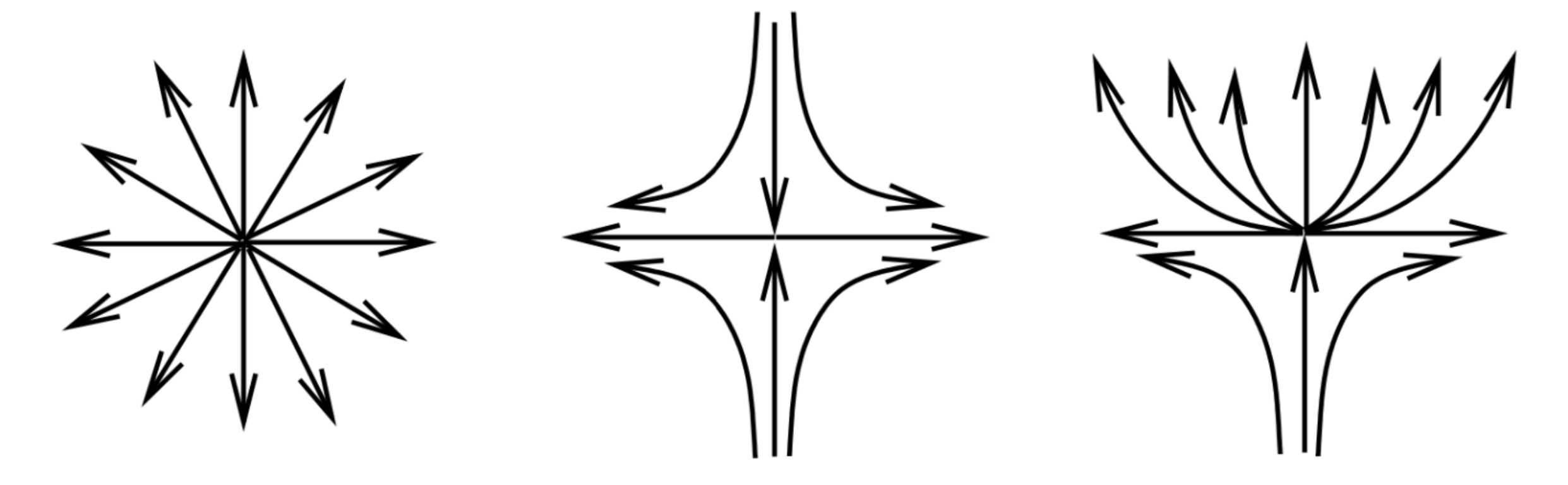}}
\caption{Elliptic, hyperbolic and embryo points} 
\label{fig:ell-hyp-embryo}
\end{figure}

Note that  a   1-parametric deformation   $Z_t= Z+  t\frac{\p}{\p y}$   has no singular points for $t>0$ and has two singular points, elliptic and  hyperbolic    for $t<0$. 
 In fact, one has the following result for any deformation of embryo singularities, see Theorem 5 in \cite{IY91}.
 \begin{prop} \label{prop:IY}
 Let   $Z_{ t}$,  $t\in\Op_{\R^n}0\subset\R^n$,   be a family of vector fields on $\Op_{\R^2}0\subset\R^2$.
 Suppose that $Z_{ 0}=   x \frac{\p}{\p x}+   y^2f(y)  ,\; f(0)\neq 0$. Then   there exist a neighborhood $U \ni 0$ in 
 $\R^2$   such that  the family $Z_t|_U$   is orbitally equivalent  for sufficiently small $t$   to a family
 $$  x \frac{\p}{\p x}+   F(y,t)\frac{\p}{\p y} ,\; F(y,0)=y^2 f(y)\neq 0. $$    \end{prop}
  Proposition \ref{prop:IY} also holds in a slightly more global  parametric form.
  \begin{prop} \label{prop:IY-param}
  Let $\Lambda$ be a compact parameter space, $\Lambda_0\subset\Lambda$ its closed subset, and $Z_{\lambda}$, $\lambda\in\Lambda$,   be a family of vector fields on $\Op_{\R^2}0\subset\R^2$ such that for all $\lambda\in\Lambda_0$  we have
  $$Z_\lambda=  x \frac{\p}{\p x}+  y^2f(y,\lambda)\frac{\p}{\p y}.$$ Then  for any $k>0$   there exist a neighborhood $U \ni 0$ in 
 $\R^2$,  a neighborhood $\Omega\supset \Lambda_0$ in $\Lambda$, and  a family of germs of $C^k$-diffeomorphisms $h_\lambda:(\R^2,0)\to(\R^2,0)$, $\lambda\in\Omega$,  such  that  
 $$ (h_\lambda)_*Z_\lambda=g_\lambda(x,y)\left( x \frac{\p}{\p x}+  F(y,\lambda)\frac{\p}{\p y}\right),\;\lambda \in\Omega,$$ where $F(y,\lambda)=y^2f(y,\lambda)$  and $g_\lambda(x,y)=1$  for $\lambda\in\Lambda_0$.
 \end{prop}
  

Recall that a contact structure $\xi$ defines an  orientation of the $3$-dimensional contact manifold. Hence, assuming the surface $S$ and the contact structure $\xi$ oriented (and hence, co-oriented),  we can distinguish between positive and negative singular points.  At a regular point $p$ of the characteristic foliation   choose a vector $\tau_S(p)\in T_pS$ which defines the given co-orientation of $\xi(p)$. Choose a vector   $Z(p)$ directing the characteristic foliation in such a way that   $(\tau_S(p), Z(p))$ defines the orientation of $S$. Near singular points this orientation is the same as defined by the Liouville field near positive points, and opposite to it near the negative ones. With this convention, positive elliptic points serve as sources and negative as sinks of the characteristic flow.

A positive embryo has  1 incoming separatrix and a half-plane filled with outgoing trajectories. We will refer to the   trajectories on the boundary of this half-plane as outgoing separatrices. For a negative embryo there are 2 incoming separatarices and one   outgoing.

 The contact structure on a neighborhood of a surface is determined by the characteristic foliation up to a contactomorphism, so the tightness can be judged by the characteristic foliation.


 For all discussions in Sections \ref{sec:char-fol}--\ref{sec:allowable} below only the topological type of the characteristic foliation will be important. In fact,  by a $C^1$-small isotopy of the surface in the ambient contact manifold, which is supported in an arbitrary small   neighborhood of a singular point, one can arbitrarily change the  smooth topology   keeping its topological type, see   Lemma \ref{lm:local-type} below. See Fig. \ref{fig:ell-standard} for the case of an elliptic point.
 
  \begin{figure}[h]
 \centerline{\includegraphics[width=10cm]{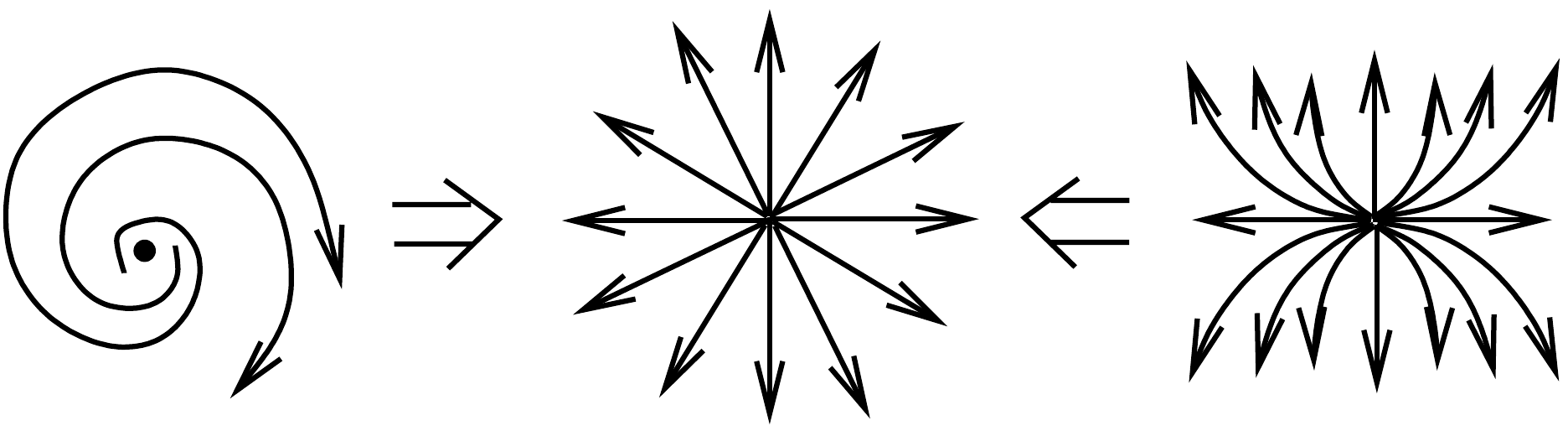}}
 \caption{Changing smooth topology of an elliptic point} 
 \label{fig:ell-standard}
 \end{figure}
 We will always picture elliptic points as nodes, see Fig. \ref{fig:ell-hyp-embryo}a) and embryos as half nodes, half saddles, see Fig. \ref{fig:ell-hyp-embryo}, though the latter picture is not possible up to   diffeomorphism.
   \begin{figure}[h]
\centerline{\includegraphics[width=4cm]{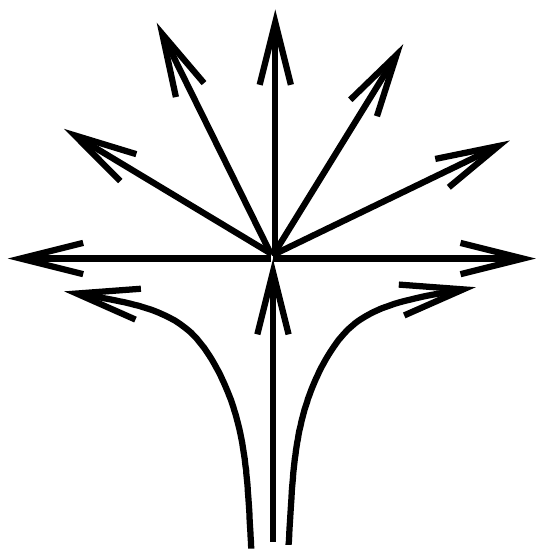}}
\caption{Topogical representation of an embryo} 
\label{fig:embryo-top}
\end{figure}
   \subsection{Manipulating characteristic foliations} 
   In the statements below by an isotopy of a surface  we always mean its isotopy in the ambient contact manifold.
   The next claim is straightforward.
\begin{lemma}\label{lm:holonomy}  
 Take a regular point $a$ of $\FF$ and  a local transverse $T$ to the trajectory $\gamma$ through $a$. Consider an isotopy $\alpha_t:T\to T$, $\alpha_0=\Id$,  supported in $\Op a$. Then  there is a $C^\infty$-small isotopy of the sphere  which realizes the holonomy $\alpha_t$ for sufficiently small $t$. See Fig. \ref{fig:perturb-fol}.

 \begin{figure}[h]
\centerline{\includegraphics[height=3.5cm,width=9cm]{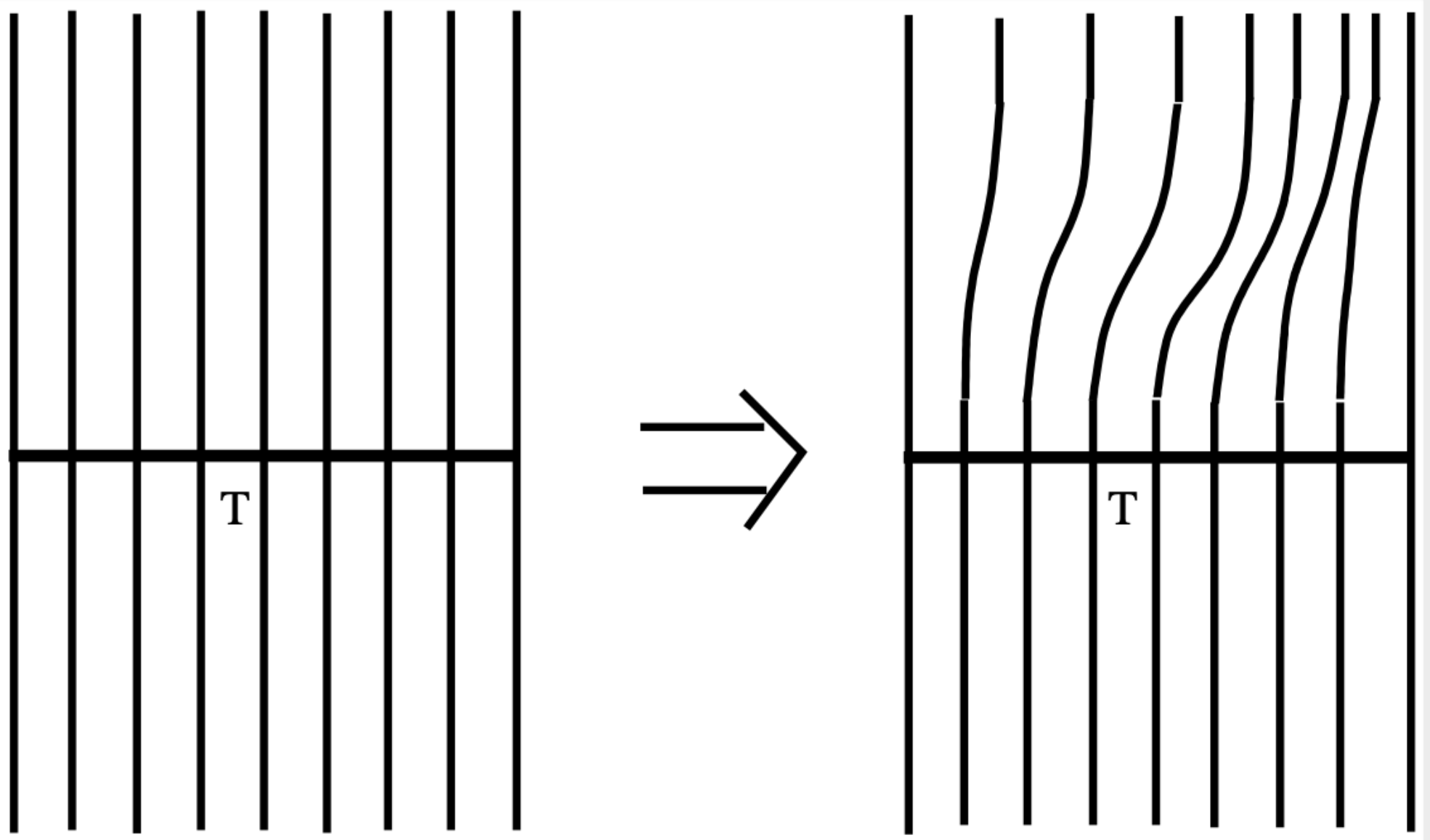}}
\caption{Perturbing a characteristic foliation} 
\label{fig:perturb-fol}
\end{figure}
  \end{lemma}

 \begin{figure}[h]
\centerline{\includegraphics[width=11cm]{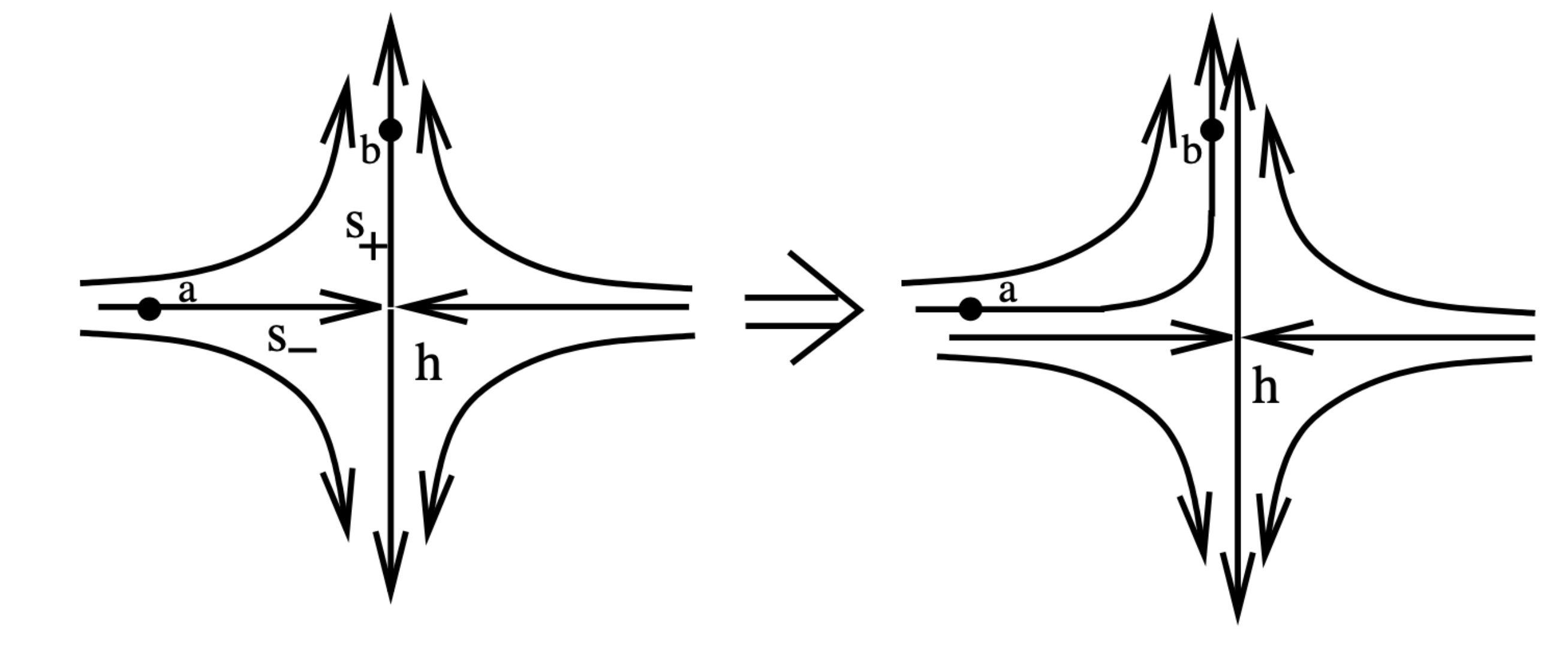}}
\caption{Bypassing a hyperbolic point} 
\label{fig:bypass-h}
\end{figure}
 Lemma \ref{lm:holonomy}  implies
\begin{cor}\label{cor:smooth-hyp} Let $h$ be a hyperbolic point, $s_-$ is an incoming separatrix, and $s_+$ outgoing. Then one can $C^\infty$-perturb $\FF$ in a neighborhood of $h$ in such a way that $s_-\cup s_+$ becomes a smooth Legendrian arc bypassing $h$. See Fig. \ref{fig:bypass-h}.
\end{cor} 
  The statement (i) in Lemma \ref{lm:killing} below is  Giroux-Fuchs  elimination lemma, see \cite{Gir91}. Other claims of the lemma are its small variations.  
\begin{lemma}\label{lm:killing}
\begin{enumerate}
\item  Let $e,h$ be elliptic and hyperbolic   points of $\FF$ of the same sign. Suppose that $e$ is of the node type. Let $\gamma$ be separatrix   of $\FF$ connecting $e$ and $h$, $\alpha$   the separatrix of $h$ opposite to $\gamma$, and  $\delta$  another trajectory ending at $e$.
Then $e,h$ can be  cancelled in such a way  that $\alpha\cup\gamma\cup\delta$ becomes a trajectory of the resulting foliation, see Figure \ref{fig:GF-elim}. The elimination can be realized by a $C^0$-small and   supported in $\Op\gamma$ isotopy of the sphere.   
  \begin{figure}[h]
\centerline{\includegraphics[width=9cm]{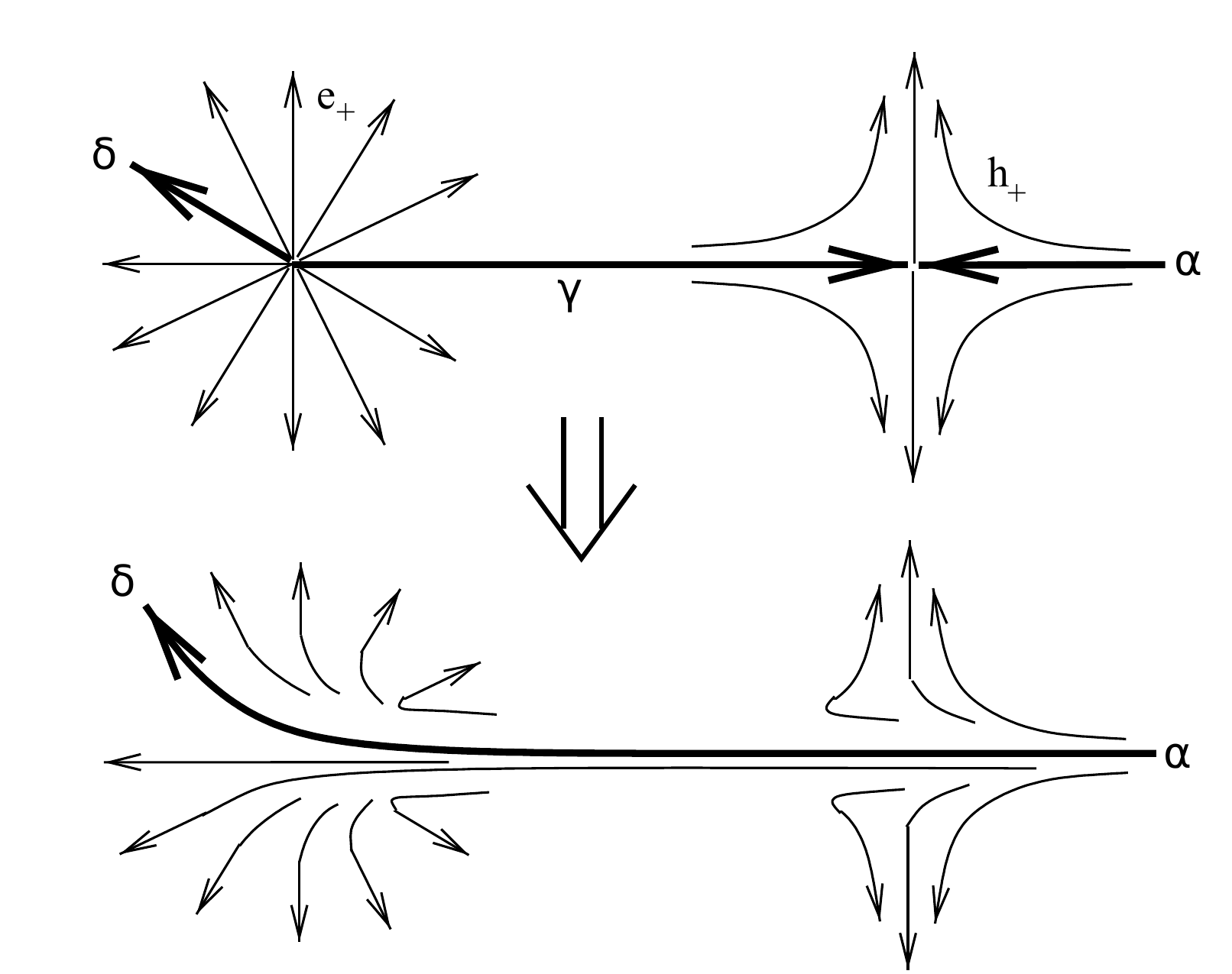}}
\caption{Eliminating an elliptic-hyperbolic pair} 
\label{fig:GF-elim}
\end{figure}
\item Let $o$ be an embryo point, $\gamma$ its separatrix, incoming  for  a positive embryo and outgoing for a negative one, and  $\delta$ any non-separatrix trajectory ending at $o$.  Then $o$ can be eliminated   by a $C^1$-small   isotopy of the sphere  supported in $\Op\gamma$.
The elimination can be done in such a way that   $\gamma\cup\delta$ becomes a trajectory of the resulting foliation, see Fig. \ref{fig:embryo-2resol1}.
\ref{fig:embryo-2resol2}.
\begin{figure}[h]
\centerline{\includegraphics[width=10cm]{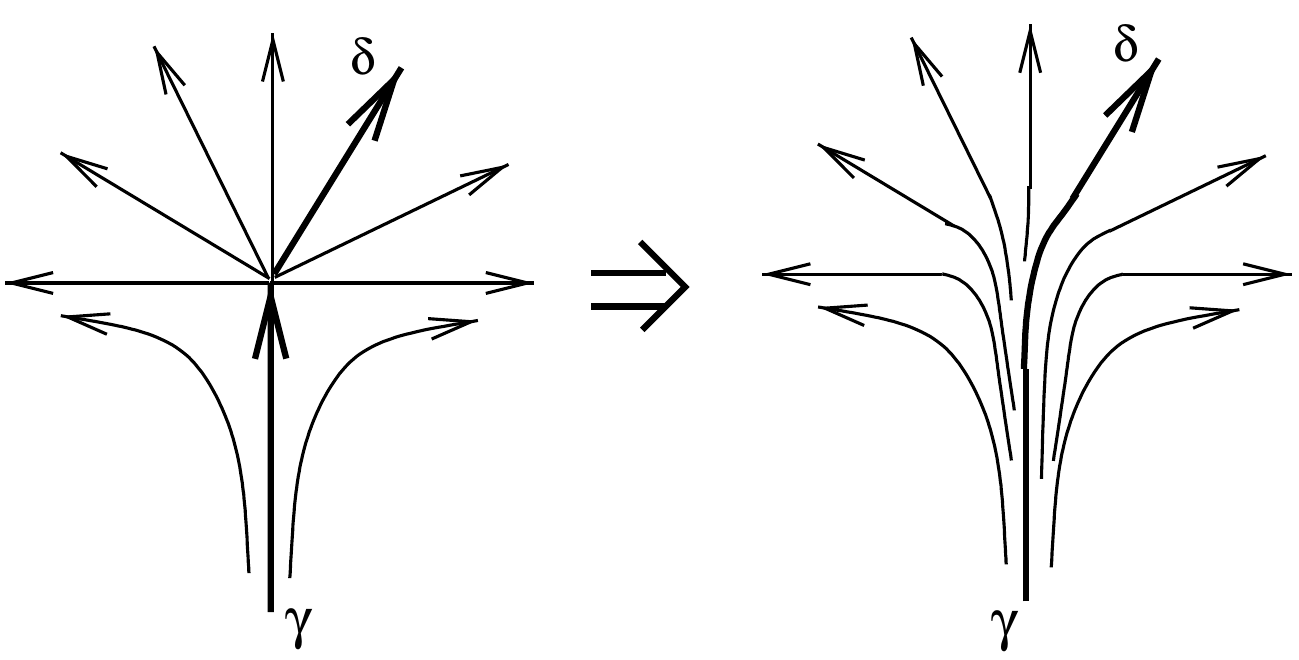}}
\caption{Eliminating an embryo} 
\label{fig:embryo-2resol1}
\end{figure}

\item Let $o$ be an embryo point,  $\delta$ any  trajectory incoming to  $o$ if $o$ is positive and outgoing if $o$ is negative. Then there exists a $C^1$-small  suppored in $\Op o$ isotopy of the sphere  which replaces   $o$ by  an elliptic-hyperbolic pair $(e,h)$ of the same sign, and    such that $\delta$ becomes one of the separatrices of $h$, see Fig. \begin{figure}[h]
\centerline{\includegraphics[width=10cm]{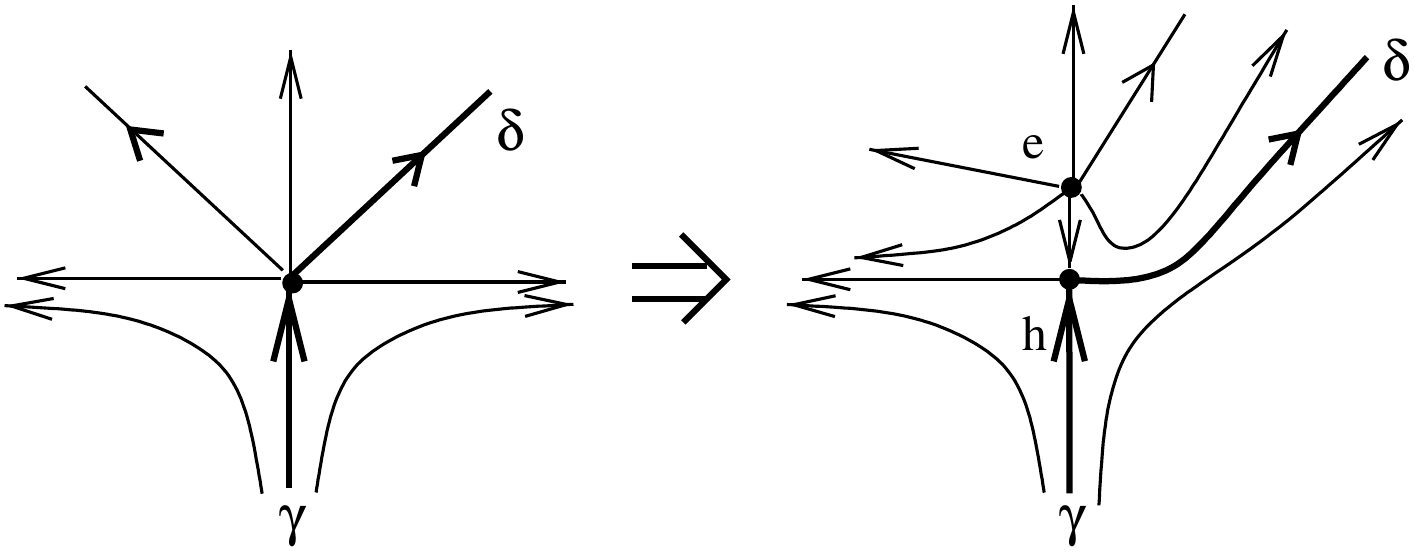}}
\caption{Resolving an embryo into an eliiptic-hyperbolic pair} 
\label{fig:embryo-2resol2}
\end{figure}

\item Let $e$ be an elliptic points and $\gamma,\gamma'$ two adjacent to it trajectories. Then by a $C^1$-small  supported in $\Op e$ isotopynone can create an elliptic-hyperbolic    pair  $e',h'$ of  the same sign as $e$ such that $\gamma,\gamma'$ becomes separatrices of $h'$, see Fig. \ref{fig:creating-e-h}.

\end{enumerate}
\end{lemma}
 \begin{figure}[h]
\centerline{\includegraphics[width=9.5cm]{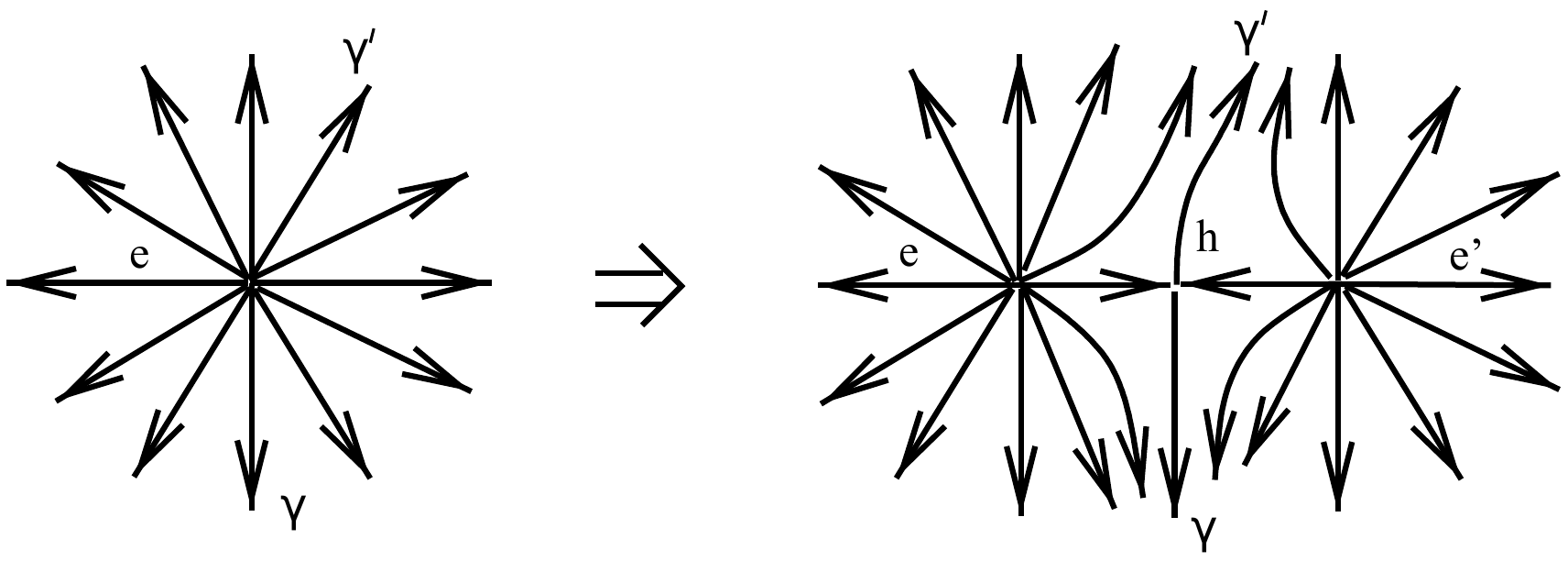}}
\caption{Creating an elliptic-hyperbolic pair} 
\label{fig:creating-e-h}
\end{figure}

\subsection{Invariants $d_\pm$}
 Let $T$  be either a closed surface in a  contact $3$-manifold,  or a surface which  bounds a  curve $\Gamma$ transverse to  the contact  structure $\xi$. In the latter case we  assume that the  oriented characteristic foliation is  outward transverse to $\p T$
and denote by $e_\pm=e_\pm(T), h_\pm=h_\pm(T)$ the numbers of elliptic and hyperbolic points. 
  Set $d_\pm:=e_\pm-h_\pm$.    
  
 We have
  $$d_++d_-=\chi(T);\; d_+-d_-=c(T),$$
  where $\chi(T)$ is the Euler characteristic and $c(T)$ is the relative Chern (Euler) number, also denoted $\ell(\p T)$ and called the  {\em self-linking} number, see \cite{El92}. It is an obstruction for extending the vector field tangent to  the foliation along the boundary $\p T$ to a non-vanishing vector field tangent to $\xi$.
  Thus,
  $$d_+=\frac12(\chi+c), d_-=\frac12(\chi-c).$$
    In particular, for a sphere $S=S^2\subset\R^3$ we have $d_\pm=1$. 

 \begin{lemma}\label{lm:d} 
 \begin{enumerate}
 \item $d_\pm(S)=1$ for a tight sphere $S$;
 \item Suppose $T$ is a genus $0$ surface with $k\geq 1$ boundary components. Suppose that $d_+(T)=1$.  Then  by  a $C^0$-small isotopy, fixed near the boundary $\p T$, one can kill all  singular points except 1  positive elliptic and $k-1$ negative hyperbolic points. In particular, when $T=D$ is a disc one can kill all  singular points except 1  positive elliptic.
 \item Let $A\subset S$ be an annulus in a tight sphere with boundary transverse in the outward sense to the characteristic foliation. Suppose $d_+(A)=1$.
 Let $D$ be  the disc bounded by  one of the boundary component $\Gamma$ of $A$ attached to 
 $\Gamma$ from the same side as $A$, see Fig. \ref{fig:annulus1}. Then $d_+(D)=1$.
 \end{enumerate}
  \end{lemma}
   \begin{figure}[h]
\centerline{\includegraphics[width=6cm]{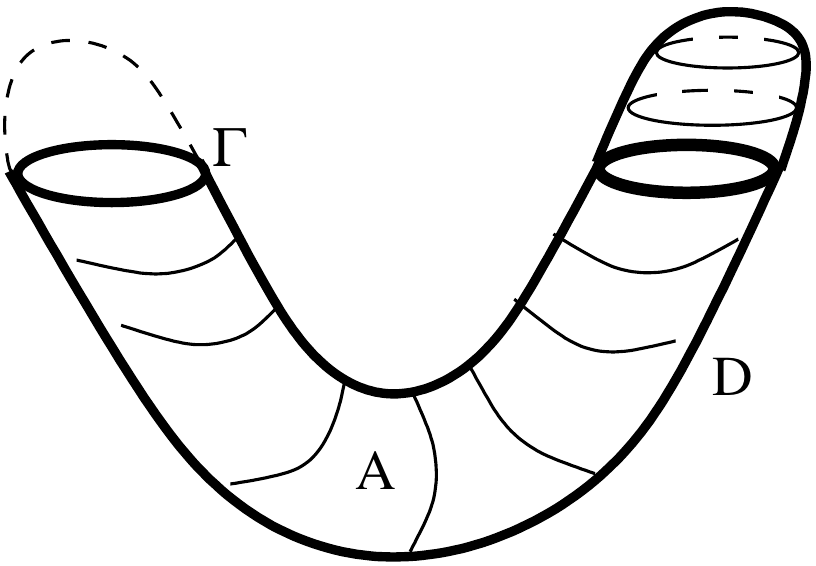}}
 \caption{Disc $D$ bounded by $\Gamma$} 
\label{fig:annulus1}
\end{figure}
  \begin{proof}
  (i) The incoming separatrices of positive hyperbolic points begin at positive elliptic points, and the outgoing separatrices of negative hyperbolic points end at negative elliptic points. Hence, all hyperbolic points can be eliminated using
  Lemma \ref{lm:killing}. On the other hand, we have $d_\pm>0$ because there should be sources and sinks. But $d_++d_-=2$, and hence, $d_\pm=1$.
  
  (ii) The incoming separatrices of positive hyperbolic points begin at positive elliptic points, and hence can be killed using Lemma \ref{lm:killing}. Suppose  that all positive hyperbolic points are killed, and hence only 1 positive elliptic left. Note that $d_-(D)=2-k-d_+(D)=1-k$, and hence there are at least $k-1$ negative hyperbolic points. If there is a negative elliptic point $e$, then  all its incoming trajectories come either from the positive elliptic point, or  from   negative hyperbolic points. If there are no incoming separatrices from hyperbolic points then  $T$ is the sphere $S$, contradicting to our assumption that $k\geq 1$. Hence, there should be a negative hyperbolic point $h$  whose outgoing separatric ends at $e$.
  Therefore, we can kill the pair $(e,h)$ using  Lemma \ref{lm:killing}.    
  
  (iii) The complement $S\setminus A$ is the union of two disjoint discs $D_1\cup D_2$. We have $d_-(D_1)+d_-(D_2)=1-d_-(A)=2$, but on the hand, each of the discs must have $d_->0$ (as the previous argument shows reversing the orientation). Hence, $d_-(D_1)=d_-(D_2)=1$ and (ii) implies that $d_+(D_1)=d_+(D_2)=0$ thus $d_+(D)=d_+(A)-d_+(D_2)=1$.
  \end{proof}

\subsection{Legendrian polygons}
 
We assume below that $\FF$ is tight and   Morse. Analogous statements hold in the generalized Morse case but we will not need them for our purposes.

A {\em polygon} is an embedded domain in $\R^2$ which is a manifold with boundary with corners.

 A {\em Legendrian polygon } in a sphere $S$ is a continuous map  $h:P\to S$ of a  polygon $P$ such that
 \begin{itemize} \item[-] $h$ is a smooth embedding on the interior  of $P$, as well as on the   boundary in the complement of vertices;
 \item[-]  each side  of $P$  is mapped onto a  leaf of a characteristic foliation on $S$ with an exception that some sides could be  mapped onto the union of two incoming or two outgoing separatrices of a hyperbolic point; in the latter  case the corresponding interior point  of the side is called a {\em pseudo-vertex}   of the polygon. See Fig. \ref{fig:Leg-polygon}.
 \end{itemize}
    \begin{figure}[h]
\centerline{\includegraphics[width=9cm]{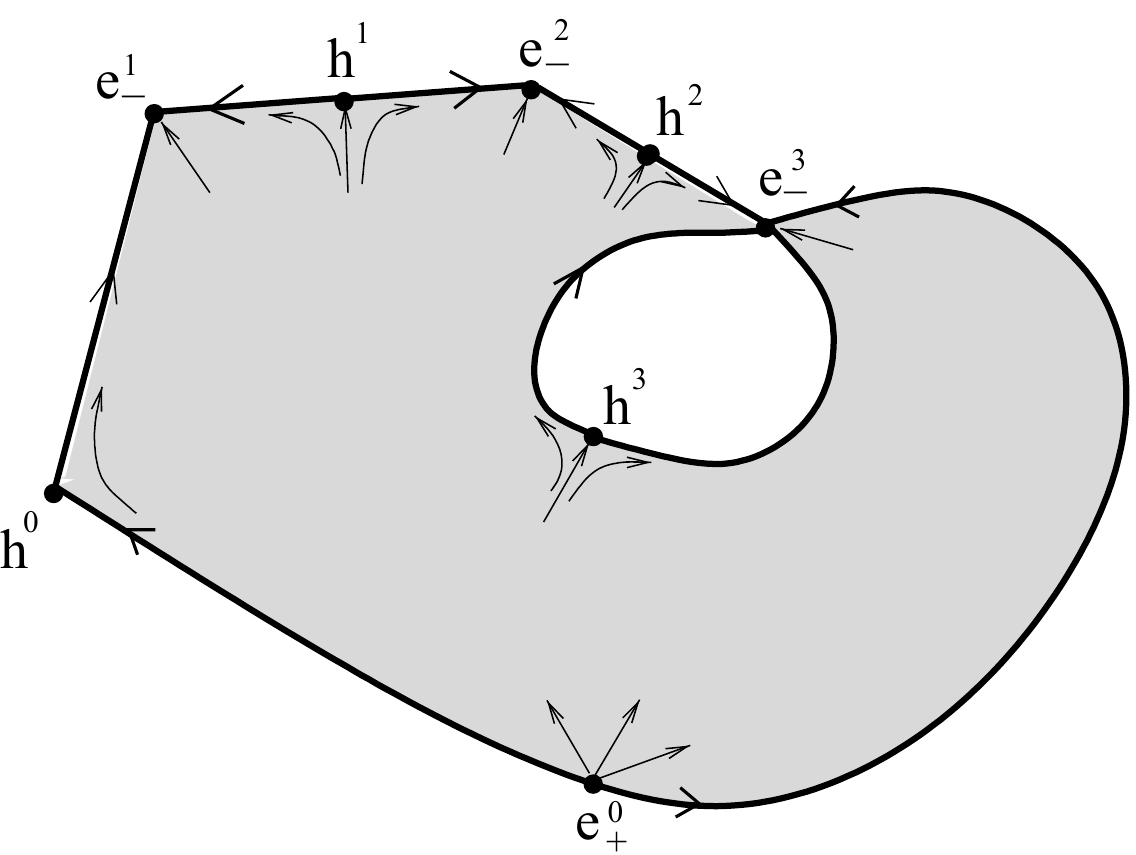}}
 \caption{Legendrian polygon; $h^1,h^2,h^3$ are pseudovertices} 
\label{fig:Leg-polygon}
\end{figure}
    \begin{lemma}\label{lm:tight-polygon}
Suppose $\FF$ is Morse and   tight. Then among elliptic  vertices and pseudovertices of a  Legendrian polygon there are points of both signs.
\end{lemma}
\begin{proof}
  If all singular points except hyperbolic  corners on the boundary of the  polygon  are of the same sign then they can first be  disjoined or smoothed using Lemma
  \ref{lm:killing}(ii) or Corollary \ref{cor:smooth-hyp}, and then pairwise cancelled  using Lemma \ref{lm:killing}(i) to get a closed non-singular leaf of the characteristic foliation. But this    contradicts the tightness assumption.
\end{proof}
   \begin{lemma}\label{lm:no-loops} Suppose $\FF$ has no homoclinics (i.e. sepratrices connecting   hyperbolic points). Then the union $\Lambda$ of stable separatrices of positive hyperbolic points 
 is a  connected tree with vertices in positive elliptic points, and edges in 1-1 correspondence with  positive hyperbolic points. See Fig. \ref{fig:eh-tree}.  \end{lemma}
     \begin{figure}[h]
\centerline{\includegraphics[width=10cm]{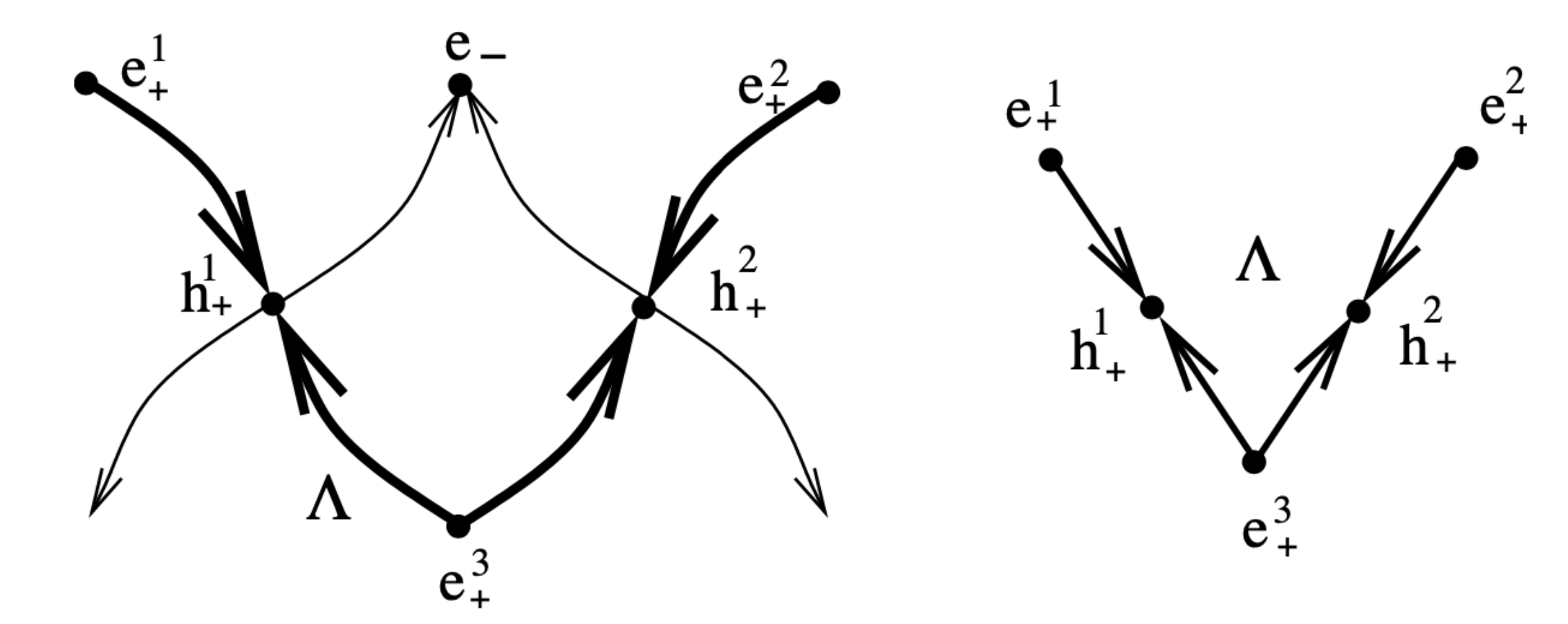}}
 \caption{Tree with vertices in positive elliptic points} 
\label{fig:eh-tree}
\end{figure}
 \begin{proof}
 First, observe that $\Lambda$ contains no loops, thanks to the tightness condition and Lemma \ref{lm:tight-polygon}.
 We also have $1=\e_+-h_+=b_0(\Gamma)$, and hence $\Gamma$ is connected.
  \end{proof}
 
\begin{cor}\label{cor:unique-path} Under the  assumptions of Lemma \ref{lm:no-loops} there is a unique path  consisting of separatrices of positive hyperbolic points connecting any 2  positive elliptic vertices. 
\end{cor}

 \section{Existence of allowable singularities}\label{sec:allowable}
 We prove in this section the main technical proposition about tight foliations on  the 2-sphere.
 \subsection{Allowable singular points}
 Given a characteristic foliation $\FF$, its singular point $x$ is called {\em allowable} in one of the following 4 cases:
 \begin{itemize}
 \item
 $x$ is a positive hyperbolic point with two  incoming separatrices from {\em different} positive elliptic points;
 \item
 $x$ is a negative hyperbolic point with two  incoming separatrices from {\em the same} positive elliptic point;
 \item
 $x$ is a positive embryo with the  incoming separatrix from  a  positive elliptic point;
 \item
 $x$ is a negative embryo with {\em all} the  incoming trajectories from  a  positive elliptic point.
 \end{itemize} 
 \begin{lemma}\label{lm:pos-psv} If two incoming separatrices of a positive hyperbolic point $h$ come from the same elliptic point then the characteristic foliation  is overtwisted.
 \end{lemma}
 Indeed,   two separatrices  form a Legendrian  polygon with only positive singular points.  

  \subsection{Key technical proposition}\label{sec:key-technical}
 The union $\Sigma=\Sigma(\FF)$ of all outgoing separatrices  of all hyperbolic points and  embryos is called the  {\em skeleton of $\FF$}.
      \begin{figure}[h]
\centerline{\includegraphics[width=12cm]{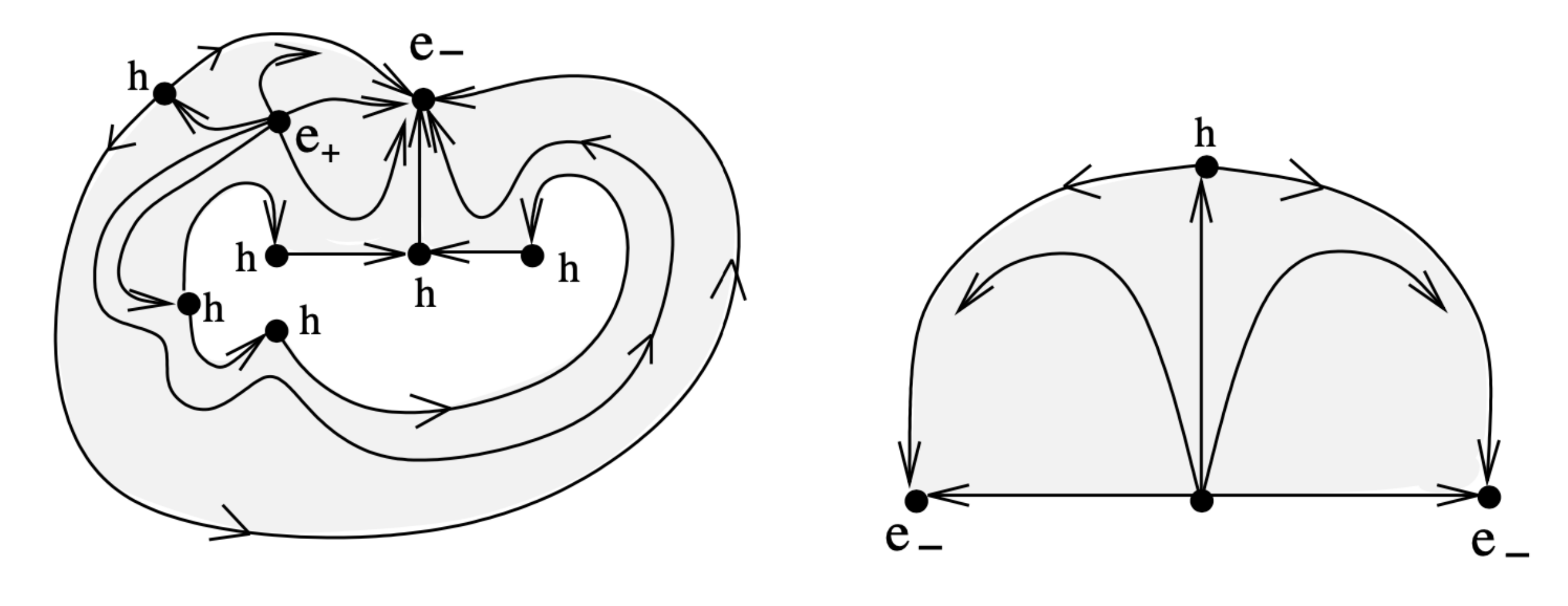}}
 \caption{Basin and semibasin} 
\label{fig:basin-semi}
\end{figure} 

 Components of $S\setminus\Sigma$  are diffeomorphic to $\R^2$ and could be of two types, {\em basins} and {\em semi-basins}.  A basin  is the union of trajectories emanating from a positive  elliptic point, called the {\em center} of the basin. A semi-basin  is the union of trajectories emanating from a positive  embryo. See Fig. \ref{fig:basin-semi}.
  \begin{prop}\label{prop:key}
 Let $\FF$ be a tight generalized Morse foliation.
 Then it has an allowable vertex.
 \end{prop}
 We begin by reducing the proposition to the case of a Morse foliation.
 \begin{lemma}\label{lm:red-to-Morse} If Proposition \ref{prop:key} holds for Morse foliations then it holds for generalized Morse foliations as well.
 \end{lemma} 
 \begin{proof}

 We argue by induction  in the number of embryos.
 Suppose the claim is proven when there are fewer than $k$ embryos.

 Let $o$ be a positive embryo. It is  {\em not} allowable if there is  either an  incoming homoclinic, or its incoming separatrix comes from another positive embryo.
  
  \begin{figure}[h]
\centerline{\includegraphics[width=11cm]{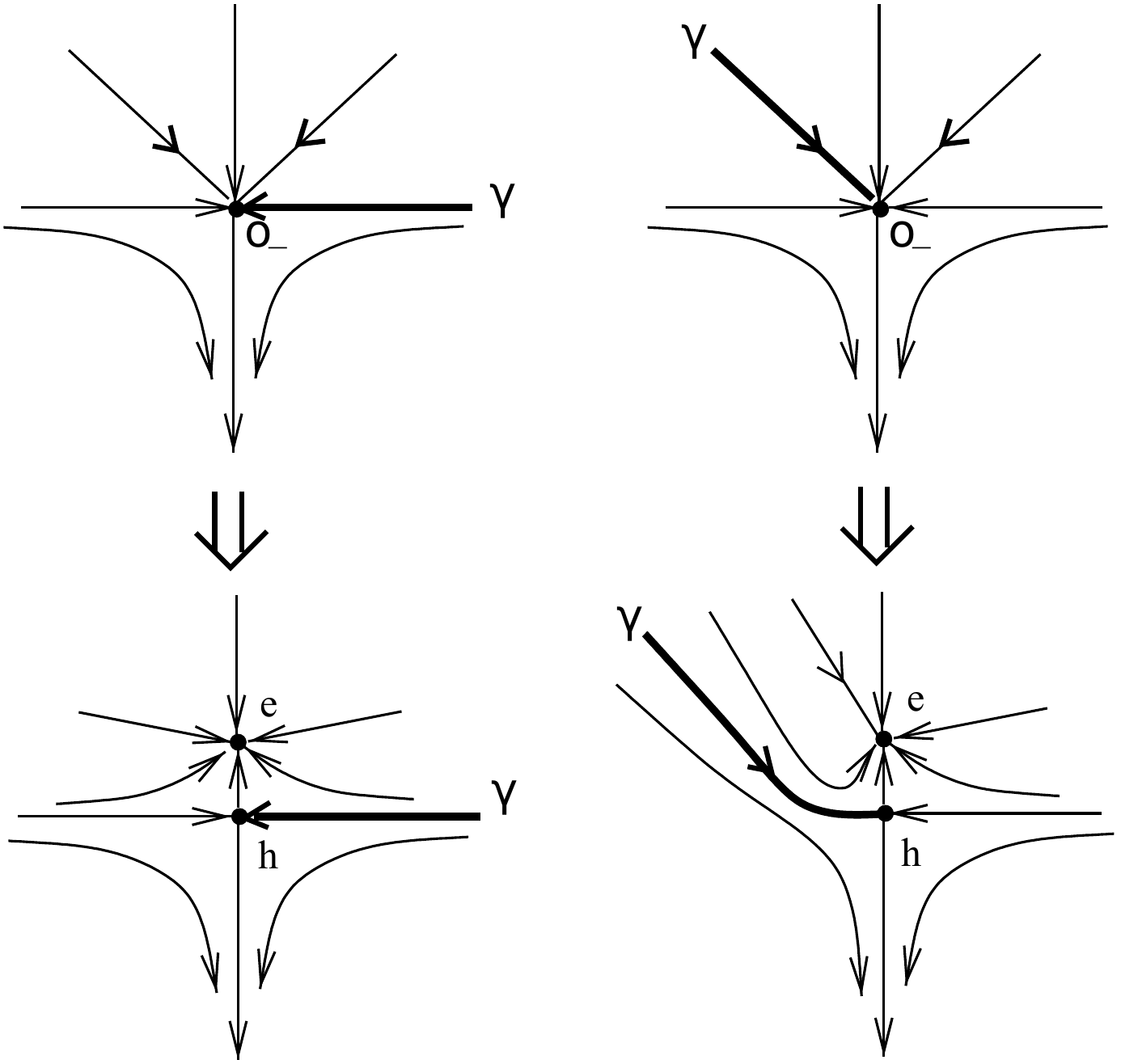}}
\caption{Resolving a negative embryo} 
\label{fig:res-neg-embryo}
\end{figure}

 In both cases let us  resolve the embryo into a pair of   positive  elliptic   and hyperbolic   points, such that the incoming separatrix  of the newly created  hyperbolic point is either homoclinic or comes from an   embryo, see Fig. \ref{fig:res-pos-embryo}. 
 This bifurcation does not create any new allowable vertices for the resulting new foliation $\FF'$, and hence, the allowable vertice provided by the  induction hypothesis for $\FF'$ is allowable  for $\FF$ as well.

 Assume now that $\FF$ has no positive embryos and let $o$ be a negative embryo.
 
 It is not allowable if and only if there is an incoming separatrix $\gamma$ either from a hyperbolic point or an embryo.
 Let us resolve the embryo $o$ into a pair of a negative  elliptic point $e$ and hyperbolic   point $h$, such that the separatrix $\gamma$ ends at   $h$, see Fig. \ref{fig:res-neg-embryo}.  Hence, $h$ is not   allowable and, therefore, the bifurcation  did not create any new allowable points, so the induction  hypothesis applies.
 \end{proof}


 \subsection{Proof of  Proposition \ref{prop:key}}
Thanks to Lemma \ref{lm:red-to-Morse}  we can  assume that $\FF$ is Morse.
Suppose that  $\FF$ has {\em   no allowable negative hyperbolic points}. 
\begin{lemma}\label{lm:bound-psv} Under the above assumption the boundary of any basin has no identified pseudovertices. \end{lemma}
Indeed, any such pseudovertex has to be negative, and hence, allowable. 

 
 \medskip
 For the induction purposes we will be proving a slightly stronger statement.
   \begin{figure}[h]
\centerline{\includegraphics[width=13cm]{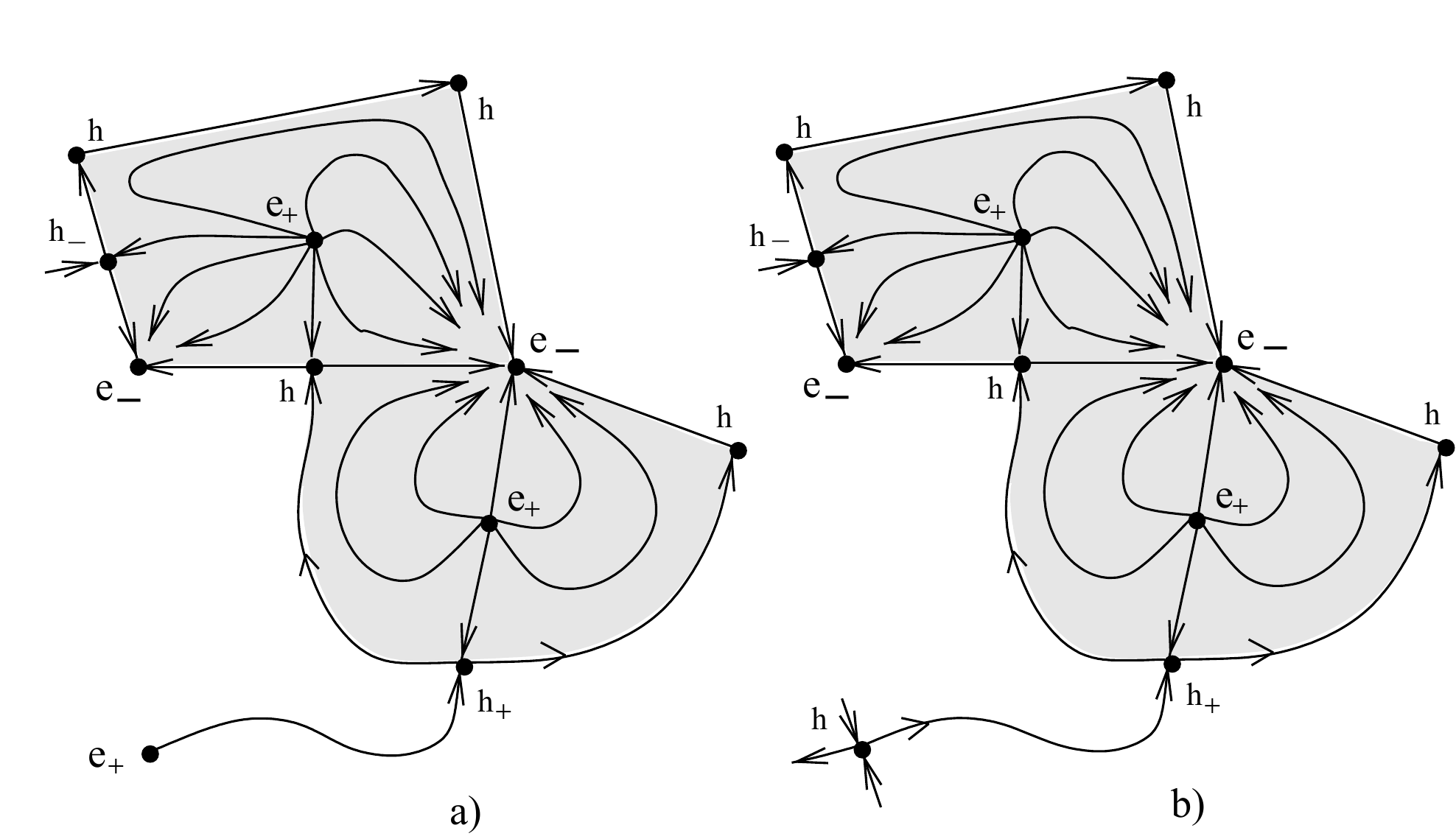}}
 \caption{Admissible (a)) and non-admissible (b)) domains} 
\label{fig:admiss-domain}
\end{figure} 
 A closed embedded domain  $U\subset S$ bounded by some of the trajectories in $ \Sigma$ is called {\em admissible} if it is a union of basins and none of positive pseudovertices on its boundary has an incoming homoclinic  from outside of
  $ U$, see Fig. 
  \ref{fig:admiss-domain}.
 \begin{lemma}\label{lm:key}
 Let $\FF$ be a tight Morse foliation. Then any admissible domain in   $U$ contains an  allowable pseudovertex.
 \end{lemma}
 \noindent{\bf Induction.} We will be proving Lemma \ref{lm:key} by induction over the total number of all singular points  and homoclinics in $U$. We assume that  {\em there are no allowable singularities in $U$} and  will deduce from that assumption that $\FF$ is overtwisted.
 
    {\sl Induction hypothesis $I_{n,k}$.}  {\em The statement holds if   there are $\leq n$ singular points and $\leq k$ homoclinics.}
    
\begin{lemma}[Base of Induction]\label{lm:base} $I_{n,0}$ holds.
\end{lemma}
\begin{proof}
 Pick any basin $T\subset U$. According to Lemma \ref{lm:bound-psv} the boundary $\p T$ has no identified pseudovertices, and hence it is a Legendrian polygon. But then Lemma \ref{lm:tight-polygon} yields a positive pseudovertex on $\p T$  which is allowable in this case.
 \end{proof}

   Suppose that $I_{m,j}$ holds for $m < n$ and all $j$, and for $m=n$ and $j<k$.  Let us prove $I_{n,k}$.
  \subsubsection*{Elimination of certain configurations} 
  
  The following sequence of claims  proves the induction hypothesis assuming existence of  certain configurations. After each step  we are adding the absence of the corresponding configuration as an additional assumption.

  {\bf Step 1.}
  {\em Suppose there is  a hyperbolic point $h$ such that  the two incoming to $h$ separatrices are homoclinic. Then $I_{n,k}$ holds.}
 
  \begin{proof}
  \begin{figure}  \label{fig:4hom}
  \centerline{\includegraphics[width=12cm]{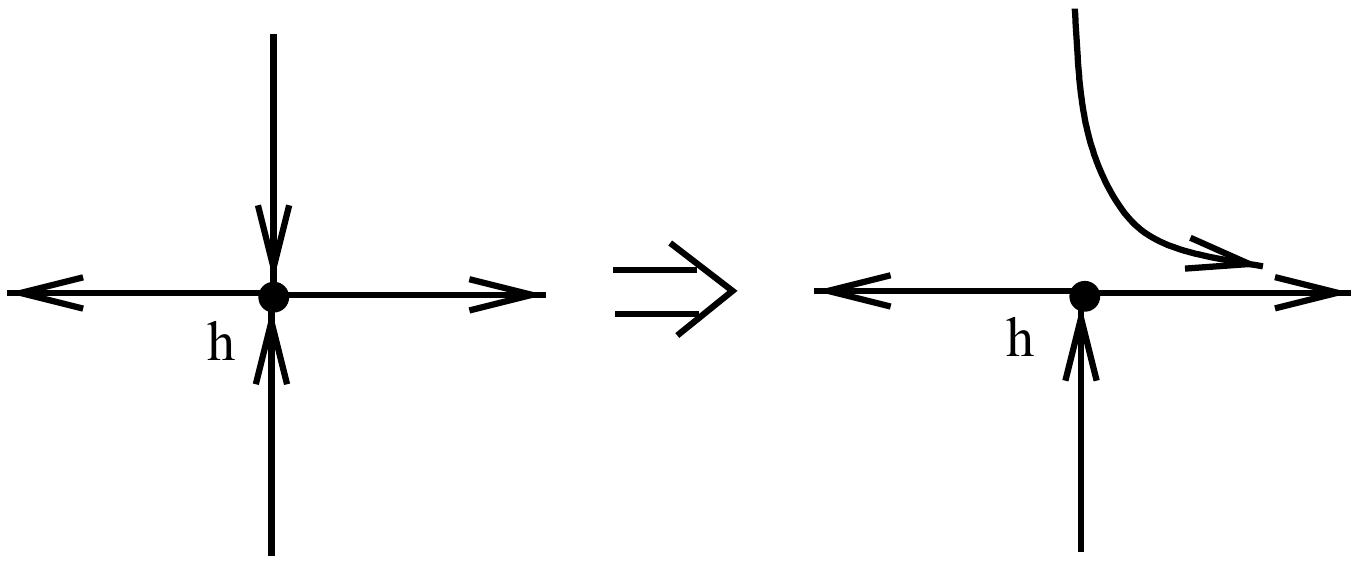}}
  \caption{Resolving  the case with 2 incoming homoclinics}
  \label{fig:resolve-4}
  \end{figure}
  Resolving one of the incoming separatrices we get a foliation without any additional allowable singularities.
  See Fig. \ref{fig:resolve-4}.
  \end{proof}
      
 Therefore,  we can assume that any homoclinic appears in a T-shaped configuration with exactly 1 incoming separatrix. We call   the hyperbolic point with the incoming homoclinic the {\em center} of the T-configuration.   The basins adjacent to the homoclinic    will be called {\em side} basins, and the third basin  will be called  the {\em base}.

 {\bf Step 2.}
{\em  Suppose that the center of a homoclinic configuration is positive, and   the base basin coincides with one of the side basins. Then  $I_{n,k}$ holds.}
 
  \begin{proof}
  \begin{figure} 
  \centerline{\includegraphics[width=9cm]{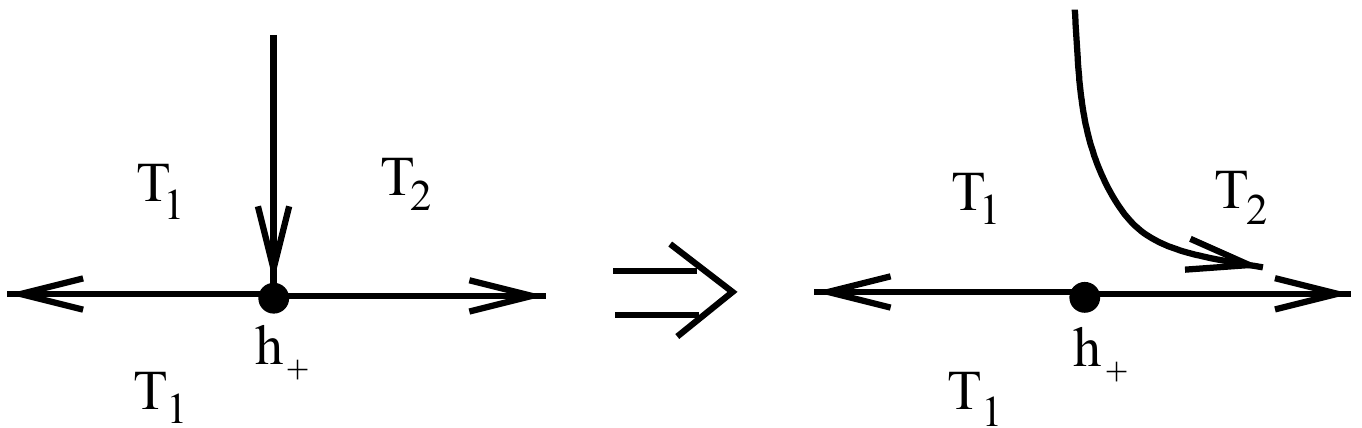}}

   \caption{1 incoming homoclinics, $h$ positive, $T_1=T_3$}
     \label{fig:pos-h-res}
  \end{figure}
  Assuming $T_1=T_3$ and 
  resolving the incoming separatrix in such a way that $h$ becomes a pseudovertex on the boundaries of $T_1$ and $T_3$,    see Fig. \ref{fig:pos-h-res}, we get
  an overtwisted foliation.  \end{proof}

{\bf Step 3.}
  {\em Suppose that any of the following two conditions holds:
  \begin{itemize}
  \item[-] the center $h$  of a homoclinic configuration is   negative, or
  \item[-] $h$ is positive but the side basins coincide.
  \end{itemize}
  Then $I_{n,k}$ holds.}
    \begin{proof}
   
  Suppose that $h$ is negative and $T_1\neq T_3$. Then
  resolving the incoming separatrix in such a way that $h$ becomes a pseudovertex on the boundaries of $T_1$ and $T_3$ we get
  a foliation without any additional allowable singularities. See Fig. \ref{fig:neg-h-res}.

   \begin{figure}
  \centerline{\includegraphics[width=9cm]{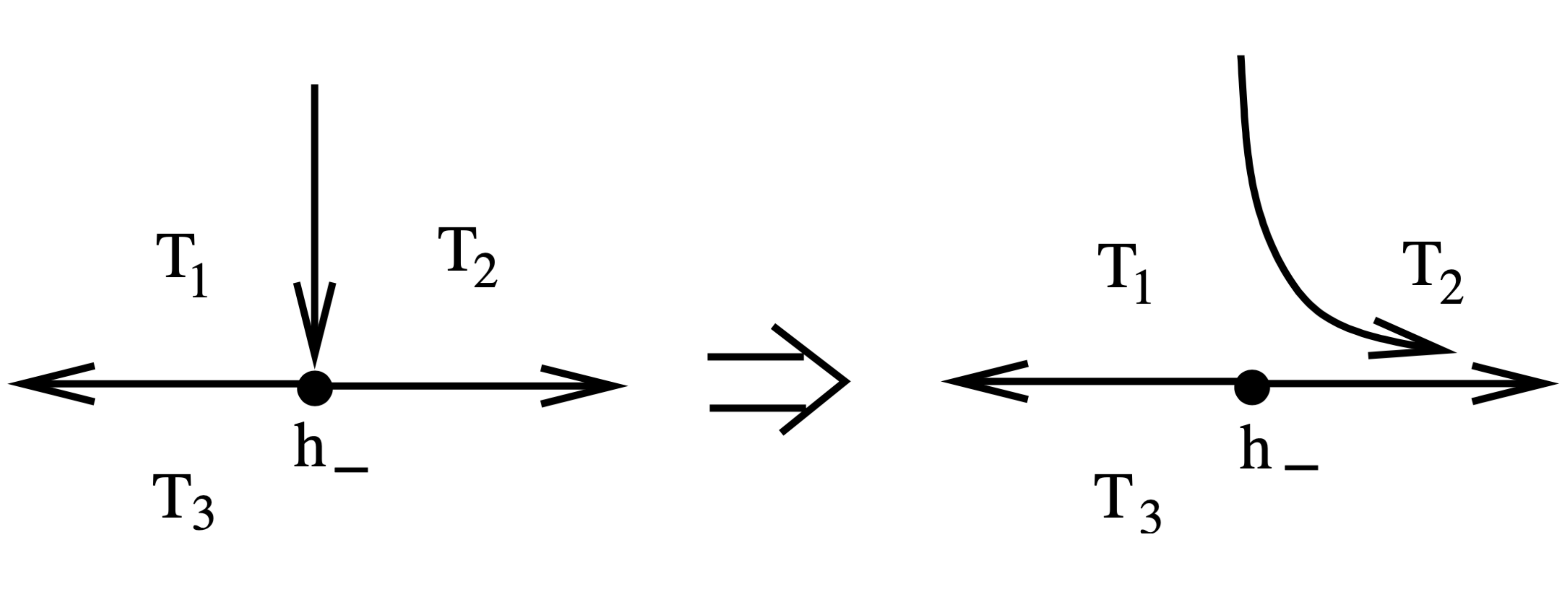}}

   \caption{1 incoming homoclinics, $h$ negative, $T_1\neq T_3$}
     \label{fig:neg-h-res}
  \end{figure}

  Hence, we can assume that for a homoclinic configuration  with a negative center  all adjacent basins coincide.
  
  \begin{figure}
  \centerline{\includegraphics[width=6cm]{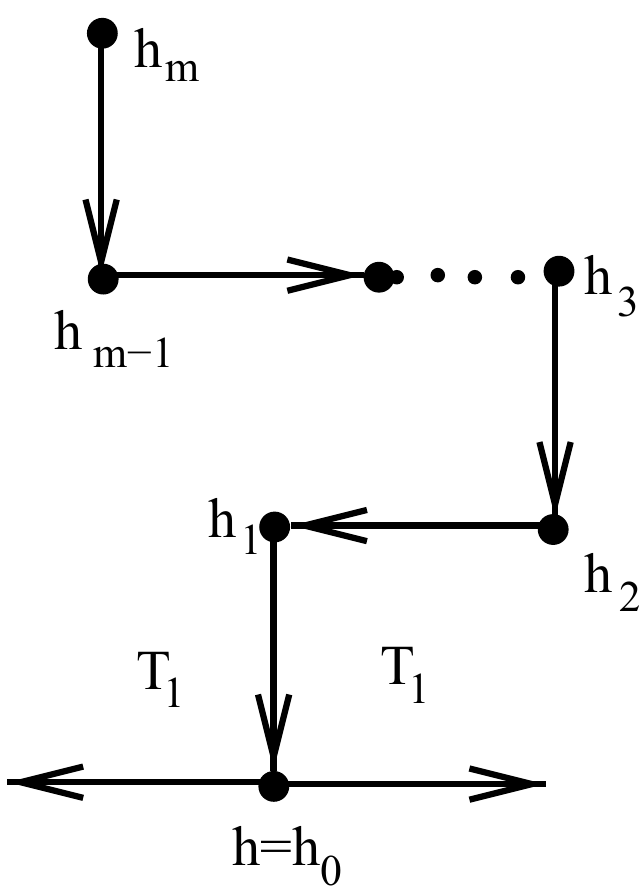}}

   \caption{Case $T_1=T_2$}
      \label{fig:chain-hom}
  \end{figure}

  Suppose   now that $h$ is either negative, or positive but   adjacent side basins coincide.
  Consider the maximal  homoclinic  chain
 $ h_m, h_{m-1},\dots, h_1,h_0:=h$ incoming to $h$, see Fig. \ref{fig:chain-hom}. If  $m=1$ then  $h_1$ is a pseudovertex.  Moreover,   it has to be negative,  because adjacent to it basins coincide, and hence,   allowable.  In the  general case $h_{1}$ still has to be negative because  we assumed that for   homoclinic configurations with positive centers  all adjacent basins  are pairwise distinct.
 Hence, arguing by induction in $m$ we conclude that  that $h_m $ is an allowable negative pseudovertex.
     \end{proof}

  {\bf Step 4.}
 {\em  If the skeleton $\Sigma$ has any end points then $I_{n,k}$ holds.}
 
    \begin{proof}
  The end point is   a negative elliptic point  $e$, and the incoming to $e$ separatrix comes from  a hyperbolic point $h$, see Fig. \ref{fig:end-point}.
    \begin{figure}
  \centerline{\includegraphics[width=6cm]{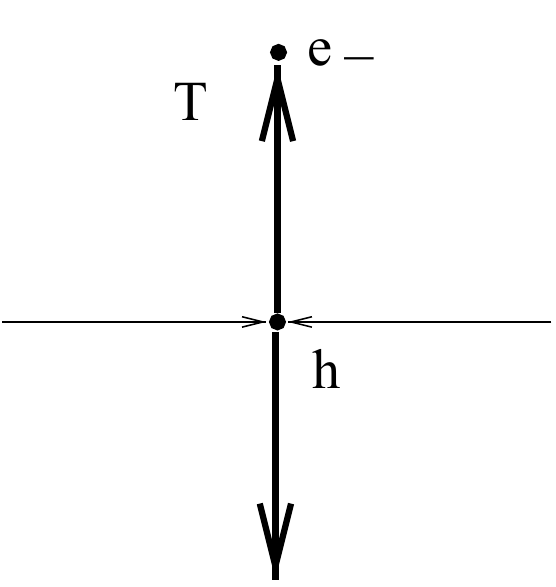}}

   \caption{Negative elliptic end point}
      \label{fig:end-point}
  \end{figure}
  We claim that  $h$ cannot     have an incoming homoclinic. Indeed, otherwise by Step 2 it would have to be negative, which was already ruled out in Step 3.
Hence, $h$ is a negative pseudovertex, as adjacent to it basins coincide, and hence it is allowable.   \end{proof}

   \begin{lemma}\label{lm:basin-psv}
   The boundary of any basin has a positive  pseudovertex.
   \end{lemma}
   \begin{proof}  By assumption there are no negative identified  pseudovertices, and no negative hyperbolic points which are centers of homoclinic configurations. On the other hand, the boundary of any basin must contain at least one pseudovertex. Indeed,  only a pseudovertex can serve as a source on the boundary of a basin.
   If all pseudovertices in   $\p T$ were negative then the polygon would be injective, and hence it would be     a Legendrian polygon, which    contradicts  Lemma \ref{lm:tight-polygon}.
   \end{proof}

  \begin{proof}  [Proof of Lemma \ref{lm:key}] Suppose that there are no   allowable pseudovertices. We will show that this assumption implies overtwistedness of $\FF$.
    \begin{figure}
  \centerline{\includegraphics[width=9cm]{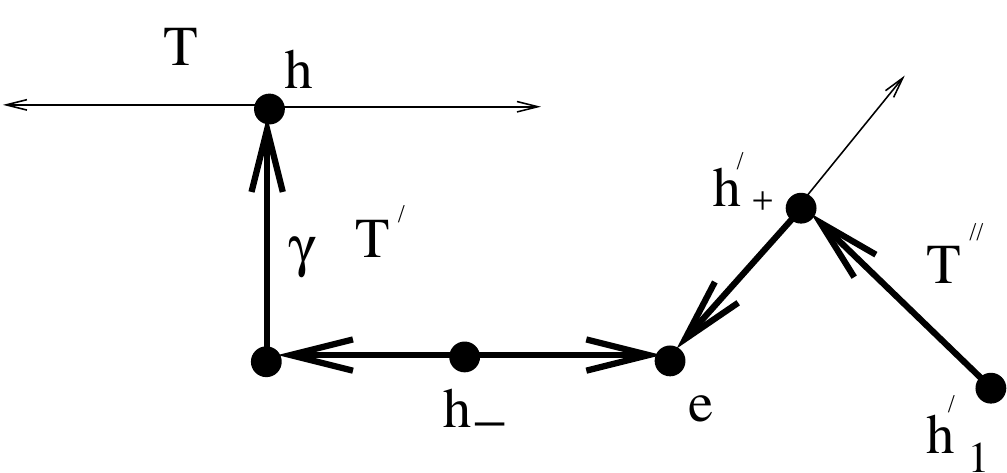}}
   \caption{Building up a path in Case $\alpha)$}
    \label{fig:building-a-path}
  \end{figure}
 Consider any basin $T\subset U$. By Lemma \ref{lm:basin-psv}    there is a pseudovertex $h\in\p T $.  It has to be positive because of Step 3. By assumption it is not  allowable, and hence, there exists   a   homoclinic  $\gamma$   incoming from outside of $T$ from a hyperbolic point  $h_1$. Consider a path $\Gamma$ from $h$ to $h_1$ along $\gamma$, and continue  it counter-clockwise along the boundary of a  side basin adjacent to $\gamma$.   There are two   possibilities:
  \begin{itemize}
  \item[$\alpha)$] we arrive to a positive pseudovertex  $h'$ with an incoming  from outside  homoclinic;
  \item[$\beta)$] the loop closes up.      
     \end{itemize}
 
 In Case $\alpha)$ we turn to the incoming homoclinic and  continue the process going counter-clockwise around the boundary of an adjacent to $h'$ side basin. See Fig. \ref{fig:building-a-path}.
  
  In Case $\beta)$ we get an    admissible domain with fewer vertices.   
  The possible configuration of closing up the loop in Case $\beta)$  are shown on Fig. \ref{fig:return-beta}.  Note that in all of the  subcases except 
  $\beta_3)$ we get an overtwisted loop.
   \begin{figure}[h]
  \centerline{\includegraphics[width=14cm]{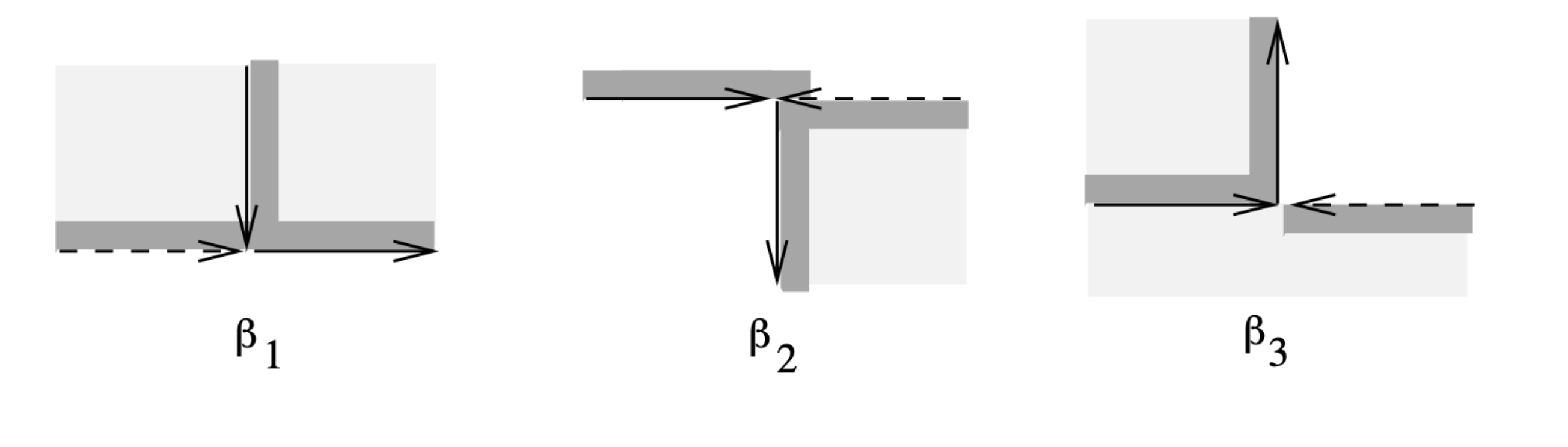}}
   \caption{Closing up a loop in Case $\beta)$}
     \label{fig:return-beta}
  \end{figure}
  This concludes the proof of Lemma \ref{lm:key}, and with it  the proof   of Proposition  \ref{prop:key}.
  \end{proof}

 \section{Simple functions on $B^3$ and $S^2$}\label{sec:simple}
\subsection{Simplicity condition} 
Consider a  function $\Phi$ on the $3$-ball $B=B^3$ without critical points,  which restricts to the boundary sphere $S=\p B$ as a generalized Morse function.
We call $\Phi$ {\em simple} if   components of each  of its level sets are  contractible.
 
 Our goal is to characterize the restrictions of simple functions to the sphere $S=\p B^3$.

 Let  $\phi:S\to\R$ be  a  generalized Morse function. 
  Choose a gradient like vector field $X$ for $\phi$ (the property we describe  does not  dependent on this choice).
  Call an index 1 Morse critical point $p$ of $\phi$ on the level $A_a:=\{\phi(p)=a\}$ {\em positive} (resp. {\em negative}) if the stable manifold of $p$ intersects the regular level set $A_a^-:=\{\phi=a-\eps\}$ at different (resp. the same) components of this level set, see Fig. \ref{fig:simple-pos-neg}.
   \begin{figure}[h]
\centerline{\includegraphics[width=8cm]{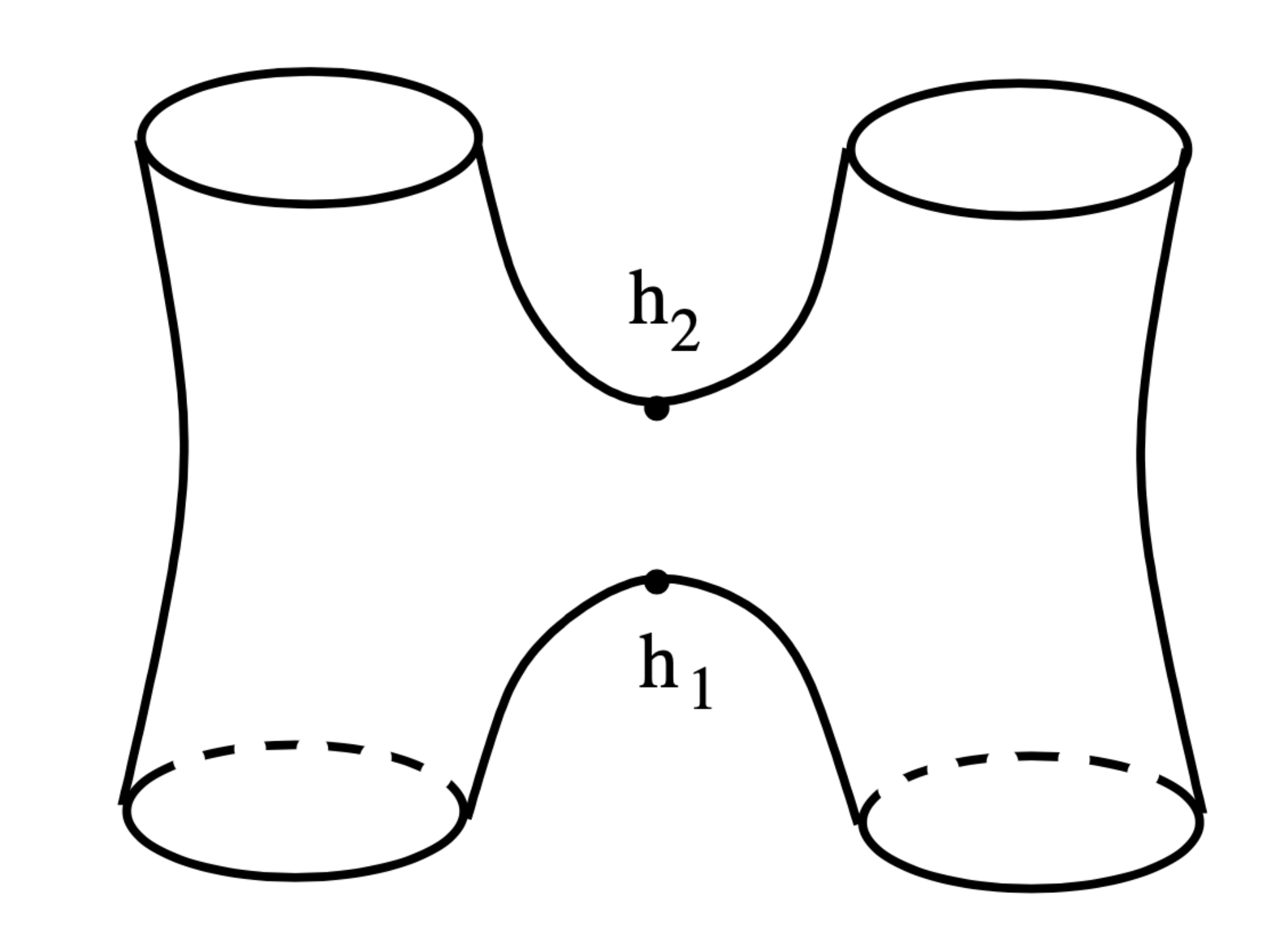}}
\caption{$h_1$ is a positive  and  $ h_2$ is a negative hyperbolic point} 
\label{fig:simple-pos-neg}
\end{figure}

Let $h^{+}_1,\dots, h^{+}_{\ell_{+}}$ be  positive  hyperbolic    points on a  critical level $A=\{\phi=a\}$. Let  $A^{-}=\{\phi=c-\eps\}$ be a regular level for a sufficiently small $\eps>0$.  Denote by $\sigma_j^\pm$ stable manifolds  of hyperbolic points $h_j^\pm$.

Consider a  ribbon graph  $\Gamma^+_a(\phi)$ whose vertices are  components of the level $A^{-}$, and     edges correspond to stable separatrices $\sigma_j^+$  of   positive hyperbolic points   on the critical level $A_a=\{\phi=a\}$, see Fig. \ref{fig:graph-Gamma}.
The ribbon structure is given by the cyclic ordering of the end points of each $\sigma_j^+$ adjacent to a given component of $A^{-}_a$. We assume the level set $A^-_a$ oriented as the boundary of $\{\phi\leq a\}$.

  A function $\phi$ is called {\em simple} if  for every critical level $a$ the graph $\Gamma^+_a(\phi)$ is a union of trees.
 \begin{figure}[h]
\centerline{\includegraphics[width=11cm]{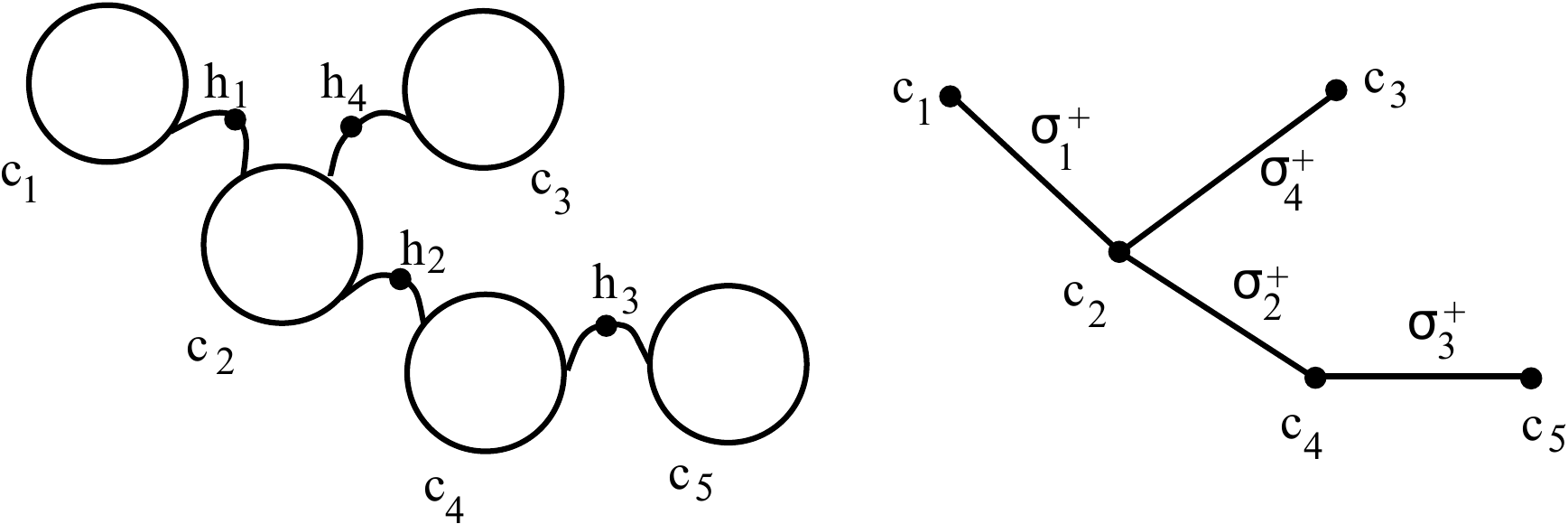}}
\caption{Graph $\Gamma^+_a(\phi)$} 
\label{fig:graph-Gamma}
\end{figure}

   Note that the simplicity condition is open, but not necessarily closed.  
  \begin{lemma}\label{lm:simple-restrict} \begin{enumerate}
  \item The restriction of a simple function from the ball to its boundary sphere is simple.
 \item Any simple function on the boundary of the ball extends to the ball as a simple function.
 \end{enumerate} \end{lemma}
  \begin{proof}
 (i) Let $\Phi:B\to\R$ be a simple function.
  Passing through a critical point of the function $\phi:=\Phi|_{\p B}$ either adds an index 1 handle to the level sets of $\Phi$, or subtracts a handle. The simplicity condition for $\Phi$ imposes constraints only on the handle addition, and  it is equivalent to the condition that the union of components of the level $\{\Phi=a-\eps\}$ with all attached handles has contractible components.
  But this means that positive points for $\phi$ corresponds to handle attaching points of $\Phi$, and the contractibility condition is equivalent to the simplicity condition for $\phi$.
  \medskip
  
(ii)   It is sufficient to consider the case when $\phi$ is a Morse function. Indeed, a creation of an embryo point is a phenomenon localized near a point, and if the function is extended to the ball $B$ it is straightforward to perform the corresponding bifurcation keeping the function on the ball critical point free.
  
  Recall a handlebody presentation for  a function $\Phi$ on a manifold $M$ with boundary which has no critical points and  which restricts to $\p M$ as a Morse function, see Fig. \ref{fig:handlebody}.  Take two copies of $M_1,M_2$ of $M$, and  consider a double $\wh M:=M_1\mathop{\cup}\limits_{\p M_1=\p M_2} M_2$, the canonical   involution $j:\wh M\to \wh M$ and a the   map $s:\wh M\to M$ with the fold on $ \Sigma:=\p M_1=\p M_2 $ which is a diffeomorphism on $\Sigma$ and interiors of the two copies. Given a function $\Phi: M\to \R$  which has no critical points and  which restricts to $\p M$ as a Morse function
  $\phi:\p M\to\R$, the function $\wh \Phi=\Phi\circ s$ has Morse critical points in 1-1 correspondence with the critical points of $\phi$. Consider the  handlebody presentation  $\wh M =\wh H_0\cup\dots\cup \wh H_m$ of $\wh M$ corresponding to the function $\wh \Phi$. The involution $j$ descends to  a handle  $H_k=D^k\times D^{n-k}$ as   
the reflection   either on the first, or the second factor. In the latter case the corresponding critical point $p$ on the critical level $a$ has the same index for $\phi$ and $\wh\Phi$ and a ``half-handle" $D^k\times D^{n-k}_+$ is attached to the  sublevel set $\{\Phi\leq a-\eps\}$ along a tubular neighborhood of $\p D^k\times 0$  in $\{\Phi= a-\eps\}$. In the former case the  critical point $p$  has  a smaller index $k-1$ for  $\wh\Phi$ and a ``half-handle" $D^k_+\times D^{n-k}$ is attached to the  sublevel set $\{\Phi\leq a-\eps\}$ along a tubular neighborhood of $\p_-( D^k_+)\times 0$  in $\{\Phi= a-\eps\}$.   Here we denoted by $D^k_+$ the upper-half disc $D^k\cap \{x_k\geq 0\}$ in $\R^k$, and by $(\p D^k_+)_-$ the part
$\p D^k_+ \cap\{x_k=0\}$ of its boundary.
   \begin{figure}[h]
\centerline{\includegraphics[width=11cm]{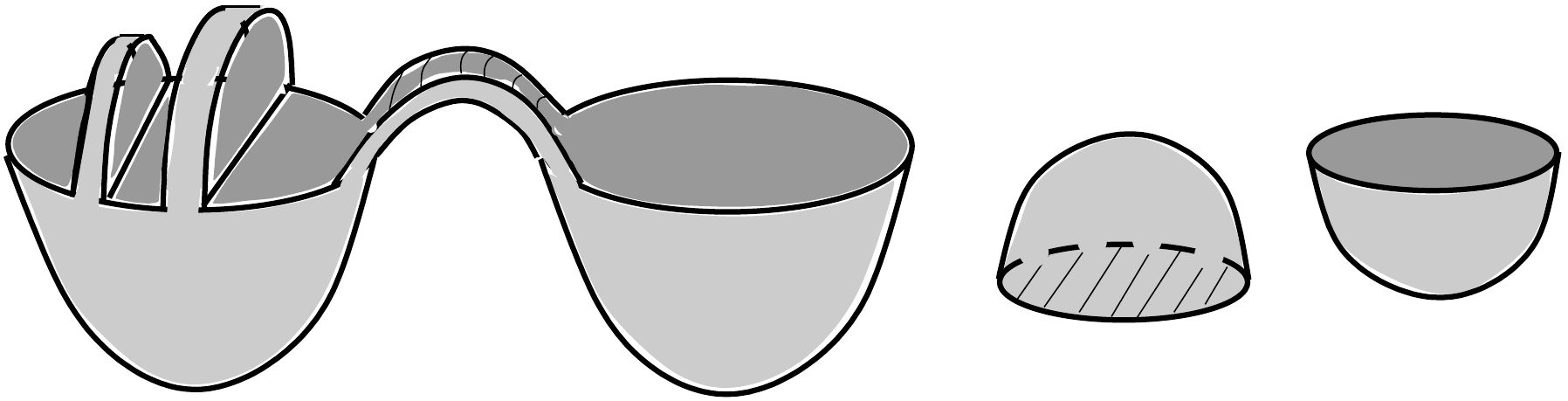}}
\caption{Handlebody presentation of a manifold with boundary} 
\label{fig:handlebody}
\end{figure}

We will be building the ball $B$ together with the function $\Phi$ inductively over critical values of $\phi$.
 
Let $ a_m$ be one of the   critical values  $a_0<\dots<a_k$ of $\phi$.
 Suppose that we already constructed a $3$-manifold  $B_m$ with boundary with corners, $\p B_m=\p_-B_m\cup\p_+B_m$, a diffeomorphism $g: \{\phi\leq a-\eps\}\to\p_-B_m$, and   a function $\Phi_m:B_m\to\R$ without critical points such that
  \begin{itemize}
 \item[-]   each component of $\p_+B_m$ is a 2-disc and $\Phi_m|_{S_+}=a-\eps$;
 \item[-] each component of $B_m$ is homeomorphic to a 3-ball;
 \item[-]   $\Phi_m\circ g=\phi$.
\end{itemize}
 For each negative hyperbolic point  $h^-_j\in A_a$ its stable manifolds  $\sigma^-_j$ have  end points on the same component $C$ of $A^-=\{\phi= a-\eps\}$.  Moreover, for any two negative hyperbolic point  $h^-_j\in A_a$ and $h^-_i\in A_a$  with the end points of $\sigma^-_j$ and $\sigma^-_i$ on $C$, these end points are not interlinked because $S$ is a sphere. 

Let $D_C$ denote  the component of $\p_+B_m$ bounded by a component $C$ of $A^-$.  For all  hyperbolic points  $h^-_{j_1},\dots,  h_{j_s}\in A_a$  with end points on  $C$ there are  disjoint embeddings      $\psi_{j_i}:\sigma^-_{j_i}\to D_C$ with end points equal to the end points   $\p \sigma_{j_i}^-\subset C$. The manifold $B_{m+1}$ is the result of attaching index $2$ half-handles $D^2_+\times D^1$ to $B_m$ along $\wt\sigma_{j_i}:=\psi_j(\sigma_j)$ and  and for each positive   hyperbolic point $h^+_j$ whose stable manifold $\sigma^+j$ have end points on components $C, C'$ of $A^{-}$ we attach an index 1-handle half-handle $D_1\times D^{2}_+$ along $\p\sigma^+_j$.  For  an index $0$ point  we add a  component 
$D^3_+$, and for an  index 2 point  we glue $ D^3_+$  to the  disc  $D_C$ bounded by the corresponding component  $C$ of $A_a$.

The simplicity condition for $\phi$ ensures that each component of $B_{m+1}$ is homeomorphic to the 3-ball, and that the canonical extension of $\Phi$ to $B_{m+1}$ is simple.

 \end{proof}    

  \begin{cor}\label{cor:symmetry} If $\phi$ is simple   then $-\phi$ is simple as well.
 \end{cor}
 Indeed, the above property obviously holds for simple functions on  the 3-ball.

   \subsection{The parametric case}
   Denote by $\Simple_3$ and $\Simple_2$ the spaces of simple functions on the 3-ball and the 2-sphere, respectively.
  \begin{prop}\label{prop:simple-extension}
  The restriction map $r:\Simple_3\to\Simple_2$ is a Serre fibration with contractible fiber. In particular,
 any family of simple functions on $S$ extends to a family of simple functions on the 3-ball $B$.
 \end{prop}

 Note that  Lemma \ref{lm:simple-restrict}(ii) asserts that the fibers of $r$ are non-empty.

 It is convenient to use the notion of a micro-fibration introduced by  M. Gromov in \cite{Gro86}. A map $p:X\to Y$ is called a micro-fibration if for any map $F:D^k\to X$ and a homotopy $f_t:D^k\to Y$ starting at $f_0:=p\circ F$ there exists $\eps>0$, and a covering homotopy $F_t$ with $p\circ F_t=f_t$ for  $t\in[0,\eps]$.
An example in Section 1.4.2 (see also  the first exercise   in Section 3.3.1) of  Gromov's book  \cite{Gro86}  states that any microfibration with non-empty contractible fibers is a Serre fibration, and in particular a homotopy equivalence. The details of the proof are  provided by M. Weiss in \cite{We05}.
 
 \begin{lemma}\label{lm:microfibration1} For any $\phi\in \Simple_2$ the fiber $r^{-1}(\phi)\subset\Simple_3$ is contractible.
 \end{lemma}
This is an immediate corollary of Hatcher's theorem \cite{Ha83}.

\begin{proof}[Proof of Proposition \ref{prop:simple-extension}]
It remains to check the microfibration property.
Let $\Phi_\lambda:B\to\R^3$, $\lambda\in\Lambda$, be any family of simple functions parameterized by a compact set of parameters $\Lambda$. The family $\Phi_\lambda$ extends to a   ball $\wh B \supset  B$ of a slightly larger radius. Let $v_\lambda$ be a family of vector fields on $\wh B$ such that $d\Phi_\lambda(v_\lambda)=1$ (e.g. $v_\lambda=\frac{\nabla \Phi_\lambda}{\||\nabla \Phi_\lambda\||})$. By integrating this vector field we define  for a sufficiently small $\eps$   an isotopy
$\theta_{\lambda,t}:B\to\wh B, \lambda\in\Lambda, t\in[0,\eps]$, beginning with the inclusion $B\hookrightarrow\wh B$ as 
$\theta_{\lambda,0}$. For any  function $\delta:  B\to[0,\eps)$ we denote by $g_{\lambda,\theta}$ an embedding $B\to \wh B$
defined by the formula $$ g_{\lambda,\delta}(x):=\theta_{\lambda, \delta(x)}.$$
Note that $\Phi_\lambda\circ g_{\lambda,\delta}(x)=\Phi_\lambda+\delta$.

Consider now  a deformation  $\phi_{\lambda, t}, \lambda\in \Lambda,t\in [0,1]$,   of the family $\phi_{\lambda,0}:=\Phi_\lambda|_{\p B}$.
For a sufficiently small $t$ we have $|\phi_{\lambda, t}-\phi_{\lambda,0}|<\eps$.
Consider a family of functions $\delta_{\lambda,t}:B\to\R$ such that for $x\in\p B$ we have  $\delta_{\lambda,t}(x):=\phi_{\lambda, t}(x)-\phi_{\lambda,0}(x)$. Then   $\Phi_\lambda\circ g_{\lambda,\delta_{\lambda,t}}(x)=\phi_{\lambda, t}(x)$ for $x\in\p B$, and hence, the family of functions $\Phi_{\lambda,t}:=\Phi_\lambda\circ g_{\lambda_t}$ provides the required extension of the family $\phi_{\lambda, t}$ to the ball $B$.  The simplicity is an open property. Hence,  the functions $\Phi_{\lambda,t}$ are simple if $\eps$ is chosen small enough.
\end{proof}
 \begin{remark} Let us  point out that the space  $\Simple_2$ (and hence, $\Simple_3$) is not contractible. In fact, it is not even simply connected, as one can exhibit a non-contractible loop of simple functions on the $2$-sphere with one minimum, one saddle point and two maxima.
 \end{remark}

\section{Taming functions and their properties}

\subsection{Lyapunov functions}\label{sec:Lyapunov}
 Recall that $\phi$ is called a {\em Lyapunov} function for a vector field $Z$ on a compact manifold $X$ if 
 $d\phi(Z)\geq c_1|Z|^2+c_2|\nabla\phi|^2$ for a positive constants  $c_1,c_2$. The property is independent of the choice of an ambient metric. Equivalently one says that $Z$ is {\em gradient-like} for $\phi$.

\begin{lemma}\label{lm:local-taming} \begin{itemize}
\item[(i)]  Let $\ZZ^*$ be the space  of  germs  of Liouville vector fields on $(\R^2,0)$ with an isolated non-degenerate zero at the origin, and $\LL$ be the space of pairs $(Z, F)$ where $Z\in\ZZ^*$ and $F$ is a  germ of   a Lyapunov function for $Z$. Then the projection $\pi:\LL\to\ZZ^*$ is a Serre fibration with a contractible fiber.  

  \item[(ii)] Let  $\Lambda$ be a compact parameter space and $\Lambda_0\subset\Lambda$ its closed subset.  Consider a family $Z_{\lambda}$, $\lambda\in\Lambda,  $      of Liouville fields on   $\Op_{\R^2}0\subset  \R^2$ such  that $Z_{\lambda }$ has an embryo point  at $0$ for  $\lambda\in\Lambda_0$. Then  there exist   a neighborhood $U  \ni 0$ in $\R^2$, a neighborhood $\Omega\supset \Lambda_0$,    and a   family    of Lyapunov  functions $\phi_{ \lambda }:U\to\R$, $\lambda\in\Omega, $ for $Z_{\lambda }|_{U}$.
  \end{itemize}
 \end{lemma}
 \begin{proof}
 (i)  Let us first show  that $\pi^{-1}(Z)\neq\varnothing$ for any $Z\in \ZZ$. If $A_\lambda=d_0Z_\lambda$ has real eigenvalues and diagonalizable in a basis $v_1,v_2$ with eigenvalues $\lambda_1,\lambda_2$ then the function  $\lambda_1x_1^2+\lambda_2x_2^2$, where $(x_1,x_2)$  are for coordinates in that basis, is  Lyapunov for $Z_\lambda$. If $A_\lambda$ has a Jordan form $\left(\begin{matrix}a_\lambda&1\\0&a_\lambda\end{matrix}\right)$, $a_\lambda=\frac12\Tr A_\lambda$, then the function $x_1^2+cx_2^2$ is Lyapunov provided that  $c>4a_\lambda^2$. If $A_\lambda$ has complex eigenvalues $\frac{a_\lambda}2\pm\alpha_\lambda i$ with       eigenvectors $v^1_\lambda\pm i v^2_\lambda$ are the corresponding eigenvectors, and $B$ is the matrix made of columns $v^1_\lambda, v_\lambda^2$ then the quadratic form  $||Bx||^2$ is Lyapunov for $Z_\lambda$.
 
Next, we observe that the fiber $\pi^{-1}(Z)$ is a convex subset  in the space of functions, and hence, contractible. Hence, it is sufficient to check a microfibration property. But this is straightforward because a Lyapunov function for $Z$ serves also as a Lyapunov function for  any  $Z'$ which is $C^2$-close to $Z$ and has an isolated singularity at the origin. 
 
 \smallskip
 (ii) According to Lemma \ref{lm:Takens-param} and  Proposition \ref{prop:IY-param} the family $Z_{\lambda}$ for $\lambda\in\Omega:=\Op_\Lambda\Lambda_0$ is orbitally equivalent to a family of vector fields
$  x \frac{\p}{\p x}+   F(y,\lambda)\frac{\p}{\p y} ,\; f(0)\neq 0,\; F(y,\lambda)=y^2f(y) $ for $\lambda\in\Lambda_0$. 
Define $\psi_{\lambda,s}=\frac{x^2}2+\int\limits_0^y F(u,\lambda)du.$    Then $Z_{\lambda}$ is a gradient vector field for  $\psi_{\lambda}$ for the standard Euclidean metric on $\R^2$.
    \end{proof}
  \begin{lemma}\label{lm:ext-to-hyp} Let  $Z$ be a vector field generating a characteristic foliation $\FF$.
\begin{itemize}
\item[a)] Let $e_1,e_2$ be positive  elliptic zeroes of $Z$, and $h$ a hyperbolic zero. Suppose that the incoming separatrices $\gamma_1, \gamma_2$ of $h$ terminate at $e_1$ and $e_2$. Let $\phi$ be a Lyapunov function for $Z$ on $\Op\{e_1,e_2,h\}$  with $\phi(h)>\max(\phi(e_1),\phi(e_2))$. Then   there exists an arbitrarily small neighborhood $U$ of  $\gamma_1\cup\gamma_2$ and a     Lyapunov function  $\Phi$ on $U$ which is constant on $\p U$ and coincides with $\phi$ in a neighborhood of critical points. See Fig. \ref{fig:ext-to-sep}a).
 \begin{figure}[h]
\centerline{\includegraphics[width=11cm]{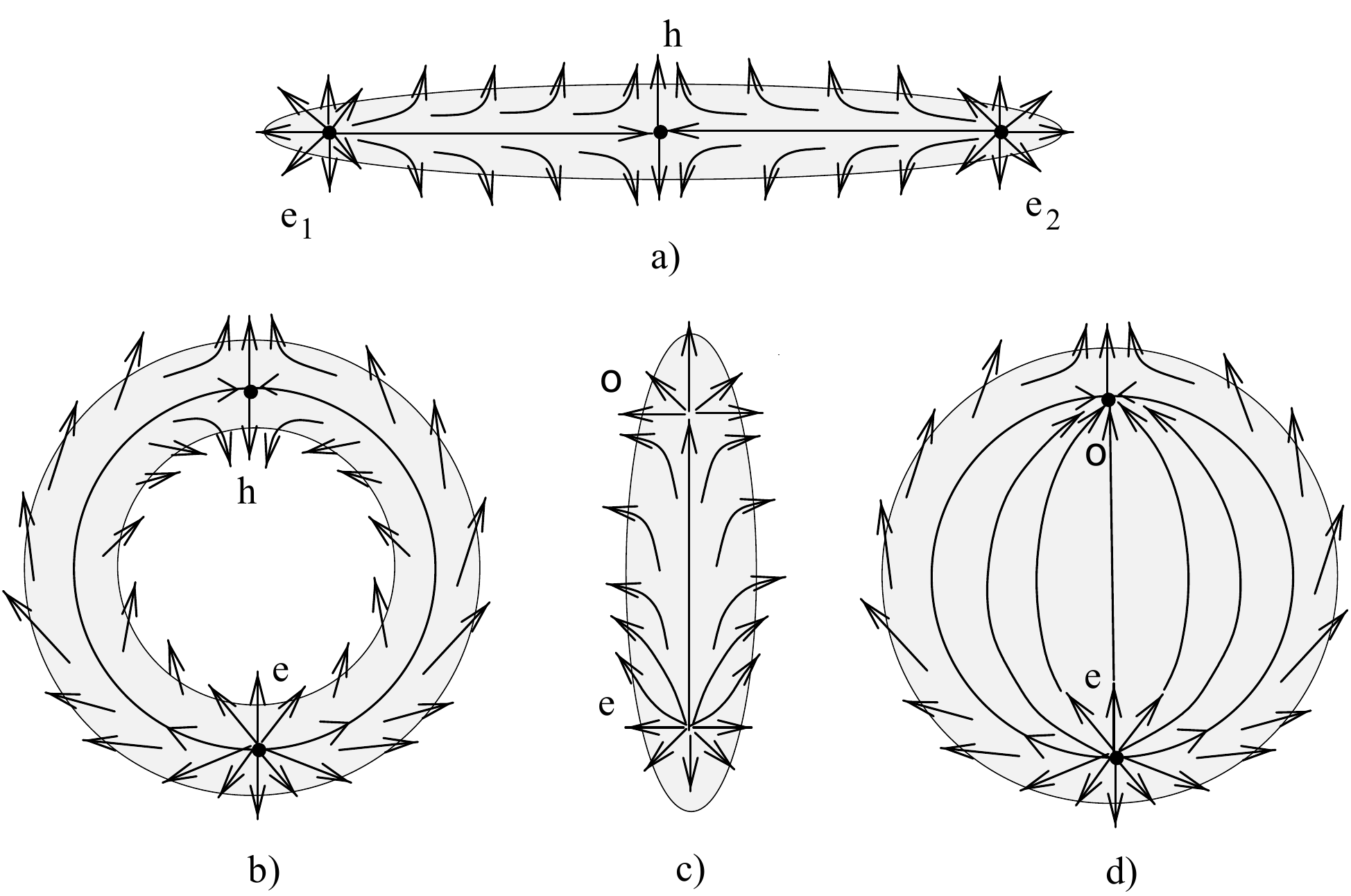}}
\caption{Lyapunov function on a neighborhood of separatrices} 
\label{fig:ext-to-sep}
\end{figure}

\item[b)] Suppose that $e, h$ are  positive   elliptic and     hyperbolic zeroes of $Z$.  Suppose that the incoming separatrices $\gamma_1, \gamma_2$ of $h$ terminate at $e$.  Let $\phi$ be a Lyapunov function for $Z$ on $\Op \{e,h\} $ with $\phi(h)> \phi(e) $.  Then   there exists an arbitrarily small neighborhood $U$ of  $\gamma_1\cup\gamma_2$ and a       Lyapunov function  $\Phi$ on $U$ which is constant on $\p U$ and coincides with $\phi$ in a neighborhood of critical points. See Fig. \ref{fig:ext-to-sep}b).

\item[c)] Suppose that $e$ is   positive    elliptic and $o$ is   positive embryo points of $Z$.    Suppose that the incoming separatrix $\gamma$ of $o$ ends at $o$.  Let $\phi$ be a Lyapunov function for $Z$ on $\Op \{e,o\} $ with $\phi(o)>\max \phi(e)$.   Then   there exists an arbitrarily small neighborhood $U$ of  $\gamma $ and a       Lyapunov function  $\Phi$ on $U$ which is constant on $\p U$ and coincides with $\phi$ in a neighborhood  of critical points. See Fig. \ref{fig:ext-to-sep}c).

\item[d)] Suppose that $e$ is positive elliptic and $o$ is a negative embryo. Suppose all incoming to $o$ trajectories originate at $e$, and $D$ is the union of these trajectories.  Let $\phi$ be a Lyapunov function for $Z$ on $\Op \{e,o\} $ with $\phi(o)>\max \phi(e)$. Then there exists an arbitrary small neighborhood $U\supset D$  and a       Lyapunov function  $\Phi:U\to\R$ which is constant on $\p U$ and coincides with $\phi$ in a neighborhood  of critical points. See Fig. \ref{fig:ext-to-sep}d).

\end{itemize}
\end{lemma}
\begin{proof} We consider only the case a); all  other cases are similar.  The construction mimics the smooth surrounding Lemma 9.20  in \cite{CE12}. Denote $\Gamma:=\gamma_1\cup\gamma_2$. First, we apply 
   We can assume that $\phi$ is equal to  $0$ at elliptic points and to 1 at the hyperbolic one, and then extend $\phi$ to $\Op(\Gamma)$ as increasing along $\gamma_j$ when going from $e_j$ to $h$, $j=0,1$. Choose an arbitrary small neighborhood $U\supset\Gamma$ on which $\phi$ is already defined.
 
  Our next goal is to construct a disk $D\subset U, \Int D\supset \Gamma$ such that $\p D$ is transverse to $Z$. Choose a sufficiently small $\eps>0$ such that the level set $\{\phi\leq \eps\}$ consists of  2 domains $\Delta_1\ni e_1,\Delta_2\ni e_2$ which are diffeomorphic to the 2-disc.
  There exists a  tubular neighborhood $\Omega\supset(\Gamma_{\eps}:=\Gamma\setminus\{\phi\leq  \eps\})$ and  its spliting $\Omega=\Gamma_{\eps}\times\Delta$, $\Delta=(-\delta,\delta)$ such that
  \begin{itemize}
  \item[-] the outgoing separatrices of $h$ form the fiber $h\times \Delta$;
  \item[-] $\p\Delta_1\cap\Omega$ and $\p \Delta_2\cap\Omega$ are fibers over the end points of the interval $\Gamma_\eps$,
   and 
   \item[-] the field $Z$ is transverse to the fibers elsewhere.
   \end{itemize}
  Hence, there are local coordinates $(s,u)\in[-1,1]\times\Delta$ in $U$ and a function $f:[-1,1]\times\Delta\setminus\{s=0,u\neq 0\}\to\R$ such that $\Gamma_{\eps}=[-1,1]$, $h=(0,0)$, and  the line field spanned by the vector field $Z$ can be given by the differential equation $$\frac{d u}{d s}=f(s,u), s\in[-1,1], s\neq 0,u\in\Delta.$$ The function $f$ satisfies   $f(s,0)=0$,  and  there exists a sufficiently small $\sigma>0$ such $f(s,u)>0$ for $|s|<\sigma, s\neq 0,u>0$ and
 $f(s,u)>0$ for $|s|<\sigma, s\neq 0,u<0$. In fact, we have 
  $\mathop{\lim}\limits_{s\to 0,u\neq 0} f(s,u)\to\pm\infty$.   Choose  a  smooth function $\wt f:[-1,1]\times\Delta\to\R$ with the following properties:
  \begin{itemize}
  \item[-] $\wt f(s,0)=0,\; s\in[-1,1]$;
  \item[-] $\wt f(s,u)=0,\; s\in[-\sigma,\sigma], u\in\Delta$;
  \item[-] $\wt f(s,u)<f(s,u),\; s\in[-1,1], u\in(0,\delta)$;
  \item[-] $\wt f(s,u)>f(s,u),\; s\in[-1,1], u\in(-\delta,0)$;
  \end{itemize}
 
Consider  a differential equation $$\frac{d u}{d s}=\wt f(s,u), s\in[-1,1], u\in\Delta,\;\;\hbox{See Fig.~\ref{fig:modif-eq}}.$$
 \begin{figure}[h]
\centerline{\includegraphics[width=13.5cm]{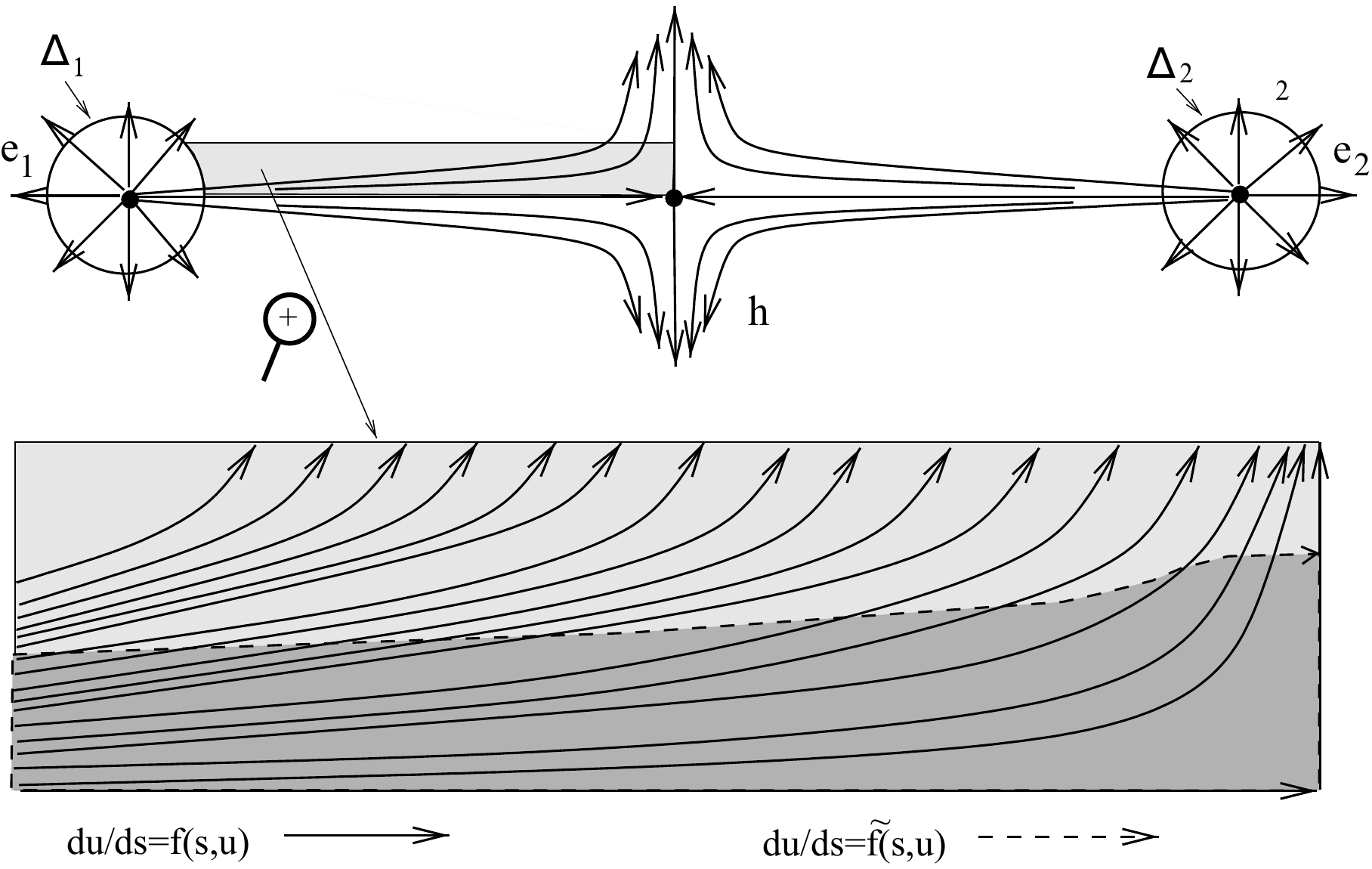}}
\caption{Solutions of differential equations  $\frac{d u}{d s}=f(s,u)$ and  $\frac{d u}{d s}=\wt f(s,u)$}
\label{fig:modif-eq}
\end{figure}
 For a  sufficiently small $\theta>0$   the equation has a solution  $u=\psi_t(s)$  with the initial condition $\psi_t(-1)=t$ if $|t|\leq\theta$, and the vector field $Z$ is outward transverse to the domain $U_\theta\subset\Omega$ bounded by graphs $u=\psi_{\pm \theta}(s), s\in[-1,1]$.  Moreover,   the graphs $u=\psi_{\pm\theta}(s)$ are  ouward transverse to  $\p \Delta_1\cup\p\Delta_2\subset\p U_\theta$.  Smoothing the corners of the domain $\Delta_1\cup\Delta_2\cup U_\theta $,
we   get the required disc with   the boundary transverse to $\Delta$.

To finish the construction  consider   a small collar $C=\p D\times[-1,1]\supset \p D, C\subset \Omega$, such that $\p D\times t$, $t\in[-1,1]$, transverse to the vector field $Z$ and such that $\p D\times (-1)\subset \Int D$, $\p D\times 0= \p D$. Consider the function  $\Phi$ equal to $\phi$ on $D\setminus C$, and to  $\max(\phi, 2t)$ on $D\cup C$. After smoothing, this function has the required properties.

\end{proof}
\begin{lemma}\label{lm:local-type}
Let $\xi$ be a contact structure on $\R^3$ such that the characteristic foliation $\FF$ on $\R^2=\{x_3=0\}\subset\R^3$  has  an isolated singularity at $0$. Let $\phi$ be a Lyapunov function for $\FF$ on $\Op 0$. Let $\FF_\loc$ be any other characteristic foliation on $\R^2 $ with an isolated  at $0$ of the same type with a local Lyapunov function $\phi_\loc$. Suppose that $\phi(0)=\phi_\loc(0)=0$.
Then for any neighborhood $U\ni 0, U\subset\R^2$, there is a   supported in $U$ arbitrary $C^1$-small 2-parametric isotopy $j_{s,t}:\R^2\to \R^3_+$, $s\in[-1,1],t\in[0,1]$, such that  
\begin{itemize}
\item[-]$j_{s,0}: \R^2\hookrightarrow \R^3$ is the inclusion for all $s\in[0,1]$;
\item[-]  $j_{1,t}(\R^2)\subset \R^3_+=\{x_3\geq 0\}, \;\;  j_{-1,t}(\R^2)\subset \R^3_-=\{x_3\leq 0\} $;
\item[-] the induced characteristic foliation $(j_{s,1})^*\xi$  
  has a unique   singularity  at $0$, 
  equals    to $\FF_\loc $ near $0$, and 
  admits a Lyapunov function $\wt\phi$ which is equal to $\phi$ near $\p U$ and to  $\phi_\loc$ near $0$. 
 
\end{itemize}
\end{lemma}
 \begin{proof}
 We can write  the contact form  $\beta$ for    the contact structure $\xi$  as $dx_3+\alpha$, where $\alpha$ is a Liouville form on $\R^2$. Let $\alpha_\loc$ be the form on $\R^2$ which  defines $\FF_\loc$.
 The pairs   $(\alpha , \phi)$ and $(\alpha_\loc,\phi_\loc)$ are  two local Weinstein structures on $\R^2$ with isolated singularities of the same type at the origin.   Using Proposition 12.12 from \cite{CE12} we can construct a Weinstein structure $(\wh\alpha,\wh\phi)$ which coincides  with $(\alpha,\phi)$  outside of  a neighborhood $U'\Subset U, U'\ni 0$ and with   $(\alpha_\loc,\phi_\loc)$ on a neighborhood $U_\loc\Subset U', U_\loc\ni 0$.  Moreover,  using Darboux's theorem we  can arrange that the symplectic forms  $d\alpha$ and $d\wh\alpha$ coincide, and hence, $\wh\alpha=\alpha +dH$, where $H(0)=0$ and $d_0H=0$.   
 By shrinking the neighborhood $U_\loc$ we can make the function $H$ arbitrarily $C^1$-small. Let $\theta:U\to[0,1]$ be a cut-off function supported in $U$ which is equal  to $1$ on $U'$. There exists $\sigma>0$ such that the function $\phi$ is Lyapunov for the    foliation on  $ {U\setminus U'}$ defined by a Pfaffian equation $ \alpha+sd\theta=0$ for any $|s|<\sigma$.    Hence, the function $\wh\phi+sd\theta$ is Lyapunov for $\wh\FF$ defined by a Pfaffian equation
  $\{\wh\alpha+s d\theta=0\}$, $|s|<\sigma$. Indeed, outside of $U'$ we have $\wh\alpha =\alpha$, and on $U'$ we have $d\theta=0$.
 Note that if $|H|<\sigma$ then the function $H-\sigma  \theta\leq 0\leq H+\sigma\theta$. We claim that  the family of isotopies $j_{s,t}$ defined by the formula.
 $$j_{s,t}(x_1,x_2):=(x_1,x_2,t(H+ s\sigma\theta)), s\in[-1,1], t\in[0,1],$$ has the required properties. Indeed, we have $j_{1,t}(\R^2)\subset\R^3_+$, $j_{-1,t}(\R^2)\subset\R^3_-$ for all $t\in [0,1]$, $j_{s,1}^*\beta=\wh\alpha+s\sigma d\theta  $, and the function  $\wh\phi$ is Lyapunov for the  characteristic foliation $\wh \FF_s=\{(j_{s,1})^*\beta=0\}$.   
 \end{proof}
 \begin{remark}\label{rm:local-parametric} Lemma \ref{lm:local-type} holds also in the relative parametric form.
 \end{remark}
\subsection{Taming functions for characteristic foliations }

A function   $\phi:S\to\R$ is said to {\em tame} the characteristic foliation $\FF$ if  the foliation can be generated by a vector field $Z$ which is gradient-like for $\phi$ and such  that its positive and negative hyperbolic points  are, respectively, positive and negative zeroes of the characteristic foliation. 

Note that local  minima  of a taming functions are automatically positive, and local maxima are negative elliptic points of $\FF$, while and positive  (resp. negative) embryos of $\FF$ are embryos the function $\phi$ of index $\frac12$ (resp. $\frac32$)  of the function $\phi$, i.e. they split into index $0$ and $1$ (resp. index $1$ and $2$) critical points under a small perturbation, see Fig. \ref{fig:simple-pos-neg-embryo}.

   \begin{figure}[h]
\centerline{\includegraphics[width=10cm]{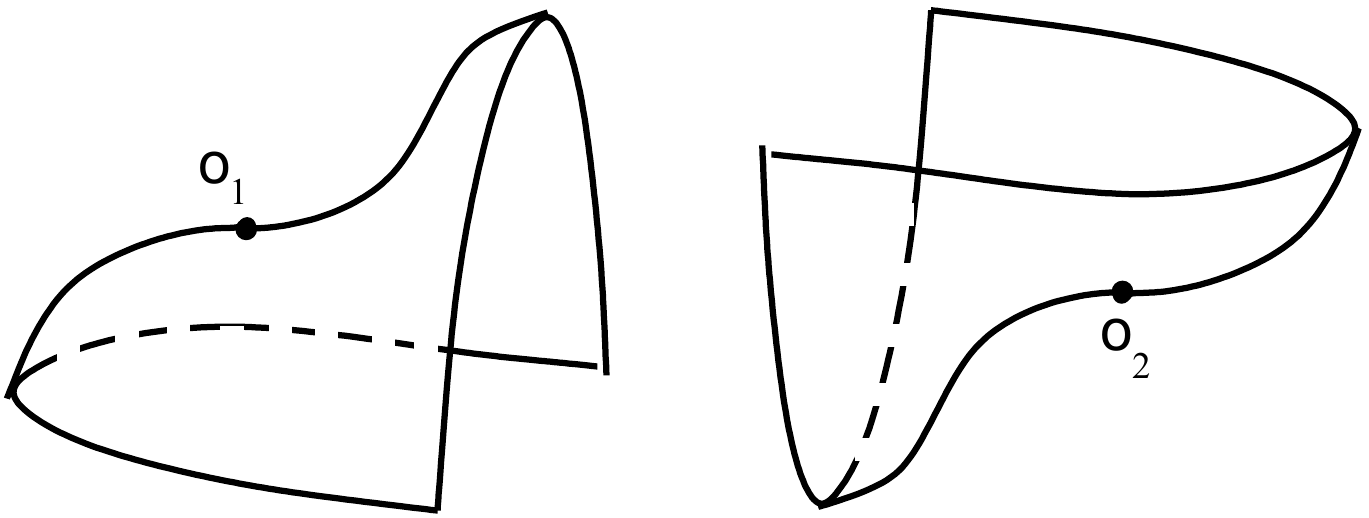}}
\caption{Embryos $o_1$ of index $\frac12$  and  $o_2$ of index   $\frac32$} 
\label{fig:simple-pos-neg-embryo}
\end{figure}

\begin{lemma}\label{lm:dplus-for-taming}
Any connected component of  a regular sublevel set of a  taming function has $d_+=1$.
\end{lemma}
\begin{proof} Arguing by induction, suppose the claim is true for sublevel  sets $\{\phi\leq a\}$ for $a< c$, where $c$ is a critical value of $\phi$. The stable separatrices of a  positive hyperbolic point  $h\in A_c$  end at two different components of $A_c^-$, and hence two different components of $\{\phi\leq c-\eps\}$. Hence, passing through $h$ connect sums two components with $d_+=1$ into   1 component with $d_+=1$.  The stable separatrices of a  negative hyperbolic point  $h\in A_c$  end at the same component of  $\{\phi\leq c-\eps\}$ and hence, passing through $h$ does not change $d_+$ of the corresponding component.
\end{proof}

  The following lemma is straightforward.
 \begin{lemma}\label{lm:taming-nearby} Suppose $\phi$ tames a Morse  characteristic foliation $\FF$. Then if $\FF'$ is sufficiently $C^2$-close to the given one and  has the same  singular points then $\phi$ tames $\FF'$ as well.
 \end{lemma}
  \begin{lemma}\label{lm:bound-basin}
 Suppose a function $\phi$ tames a   tight foliation $\FF$ without homoclinic connections between hyperbolic points or embryos.
 Given any two positive elliptic points $e_+, e_+'$, let   $\gamma(e_+, e_+')$  be the unique Legendrian path $e_+h_+^1e_+^1\dots h_+^ke_+^k$ provided by Corollary \ref{cor:unique-path}. Denote   $$c(e_+,e_+';\phi)=\min\{\phi(h_+^j), h_+^j\in\gamma(e_+, e_+')\}.$$
 We set $c(e_+,e_+;\phi)=0$ (assuming that $\phi=0$ at positive elliptic points). For each hyperbolic point  we denote by $E(h)=\{e_+,e_+')$ the positive elliptic end points of its stable separatrices.
 Then for any negative hyperbolic point $h_-$ of $\FF$ we have
\begin{align}\label{eq:simplicity}\phi(h_-)>c(e_+,e_+';\phi),\;\; \hbox{where}\;\;E(h_-)=\{e_+,e_+').\end{align}
\end{lemma}
\begin{remark}\label{rem:taming-no-hom}
Note that if $\phi$ tames a   tight foliation $\FF$ with  homoclinic connections, then it also tames a   tight foliation $\FF'$ without  homoclinic connections, because both properties, taming and tightness, are   open.
\end{remark}
\begin{proof} 
Let $E(h_-)=\{e_+,e_+')$. If $e_+=e_+'$ then the statement is vacuous, so we suppose that $e_+\neq e_+'$.  For any  regular value $c\in(0,\phi(h_-))$ denote by $C_{c}, C_c'$ components  of the level set $A_c$ intersecting separatrices of $h_-$ originated at $e_+$ and $e_+'$, respectively. For $c$ close  to $\phi(h_-)$ we have  $C_c=  C_c'$  while  for 
$c<c(e_+,c_+')$ we must have  $C_c\neq  C_c'$.
\end{proof}
 
 \begin{lemma}\label{lm:2-pos}
 
 Let $\FF$ be a tight characteristic foliation without homoclinic connections. Then any Lyapunov function for $\FF$  which satisfies condition \eqref{eq:simplicity} is taming.   \end{lemma}
 
 \begin{proof}  Let us  check that condition  \eqref{eq:simplicity} implies that the two notions of positivity coincide for all hyperbolic points.  Suppose  that $h$ is  negative for the function, i.e.   both incoming separatrices of $h$ intersect the same component    $C\subset A_c^-$, and hence the same component $U$ of $\{\phi\leq c-\eps\}$.  Arguing by induction, we can assume that the claim is already proven for hyperbolic points in $\{\phi<c\}$, and hence Lemma \ref{lm:dplus-for-taming} implies that $d_+(U)=1$  and hence,  we can use Lemma \ref{lm:d}(ii)  to  eliminate  all critical points in $U$ except  1 positive elliptic point. Then both separatrices of $h$  will be   originating at the same positive elliptic point which by Lemma \ref{lm:no-loops} implies that $h$ is negative for $\FF$. If $h$ is positive for $\phi$ then incoming separatrices of $h$ intersect different  components    of  $  A_c^-$, and hence, taking into account that $S$ is the sphere,  different  components  $U,U'$  of $\{\phi\leq c-\eps\}$.     Hence, generically   the stable separatrices   of $h$ originate at positive elliptic points $e\in U$ and $e'\in U'$. Let  $\gamma(e,e')$ be the path provided by Corollary \ref{cor:unique-path}. If $h$ were negative for $\FF$ then we would have $\phi(h)\leq c(e,e')$, contradicting  condition \eqref{eq:simplicity}. Hence, $h$ is positive for $\FF$.

   \end{proof}
  \begin{lemma}\label{lm:tame-simple}
  Any taming function is simple.
  \end{lemma}
\begin{proof}
Suppose there is a  loop  in  the graph $\Gamma_+^c$  for a critical value $c$. Using  Lemmas \ref{lm:dplus-for-taming} and \ref{lm:d}(ii) we deform $\FF$  on components  $\{\phi\leq c-\eps\}$ to leave exactly one positive elliptic point in each component. But then the separatrices  of positive hyperbolic points  forming the cycle could be continued to these elliptic points to create a Legendrian polygon with only positive vertices,
 thus contradicting Lemma \ref{lm:no-loops}.
\end{proof}
 

 \subsection{Existence of taming functions}
 
 \begin{prop}\label{prop:tight-2-sphere} Let $(S,\FF)$ be  a  generalized Morse  tight  foliation on  a 2-sphere. Then  $S$ admits a taming function.
\end{prop} 
 
\begin{remark} As we will see in Proposition \ref{cor:taming-tight} below existence of a taming function for a characteristic foliation is not only necessary, but  also sufficient for tightness.
\end{remark}
 
 \begin{lemma}\label{lm:main}
 Let $\FF$ be a tight foliation on $S$.  Suppose that $\FF$ has non-elliptic critical points.
 Then there exists a function $\phi:S\to\R$ which has the following properties:
 \begin{itemize}
 \item [-] $\phi$ attains its minimal value $0$ at all  positive elliptic points;
 \item [-] $\phi$ tames $\FF$ on $V:=\{\phi\leq 1\}\subset S$. There is a unique critical point $c$ in
 $\{0<\phi\leq 1\}$;
 \item [-] $\phi(c)=\frac12$ and $c$     is either hyperbolic or embryo.
 \end{itemize}
 Moreover, we can arbitrarily up to additive constants prescribe the Lyapunov function $\phi$ near the critical points of $\phi$ in $\{\phi\leq 1\}$.
 \end{lemma}
 \begin{proof} Lemma \ref{lm:main} is essentially  a reformulation of   the Key Proposition \ref{prop:key}. First,   let us    construct the taming function   with the critical value $0$ on  neighborhoods of positive elliptic points, see Lemma \ref{lm:local-taming}. According to Proposition  \ref{prop:key} there exists an allowable vertex, which is either
 \begin{itemize}
 \item[a)] a positive hyperbolic point with two  incoming separatrices from   different  positive elliptic points, or
 \item[b)] a negative hyperbolic point with two  incoming separatrices from  the same  positive elliptic point, or
 \item[c)] a positive embryo with the  incoming separatrix from  a  positive elliptic point, or
 \item[d)] a negative embryo with {\em all} the  incoming trajectories from  a  positive elliptic point.
 \end{itemize}
In case a) let  $h$  be a positive hyperbolic point with the  incoming separatrices from two positive elliptic points. Then according to Lemma \ref{lm:ext-to-hyp}a) there exists a function on a neighborhood $U$ of the union of both separatrices which has the elliptic points as their minima and  the hyperbolic point as the saddle point and which is constant on the boundary $\p U$. Similarly, in case b)  let $h$ be  a negative hyperbolic point with the  incoming separatrices from the same positive elliptic point.  Then Lemma \ref{lm:ext-to-hyp}b) yields a taming  function on a neighborhood $U$ of the union of both separatrices which has the elliptic point as its minimum and  the hyperbolic point as the saddle point and which is constant on the boundary $\p U$. In both cases the simplicity condition is satisfied. Cases c) and d), of a positive  and negative embryos   follow from  Lemma \ref{lm:ext-to-hyp}c) and d).
In all the above cases we can prescribe $\phi$ to be equal (up to an additive constant) to any Lyapunov function. \end{proof}
 
 \begin{proof}[Proof of  Proposition \ref{prop:tight-2-sphere}]
 We prove  Proposition \ref{prop:tight-2-sphere}  by induction in the number of critical points.
Suppose   the proposition is already proved  for less than $k$ hyperbolic points or embryos.
Consider the function $\phi$ provided by Lemma \ref{lm:main}.
If $c$ is a positive hyperbolic point then the simplicity condition implies that each component of $V=\{\phi\leq 1\}$ is a disc with $d_+=1$. The same obviously holds if $c$ is an embryo. Hence, Lemma \ref{lm:d}(ii) allows us to kill all critical points in these discs except 1 positive elliptic. By induction hypothesis the new tight foliation $\FF'$ which coincides with $\FF$ on $\{\phi\geq 1\}$ admits a taming function $\psi$. We can assume that  $\{\psi\leq 1\}= V$. Hence,   concatenating    $\phi|_V$  with $\psi|_{S\setminus V}$ we get the required taming function for $\FF$. If $c$ is a negative hyperbolic point, then one of the components of  $V$ is an annulus $A$. The complement of $A$ is the union of two discs, $S\setminus \Int A=D_1\cup D_2$. For each of the complementary discs, $D_j':=S\setminus\Int D_j, j=1,2$ we have $d_+(D_j')=1$, according to Lemma \ref{lm:d}(iii). Hence, we can deform away all singular points but 1 positive elliptic point on $D_1'$, and using the induction hypothesis construct a  function $\psi_1$  which tames the foliation $\FF_1$ on $S$ which is equal to $\FF$ on $D_1$ and to the foliation with 1 elliptic singularity on $D_1'$. Similarly, we construct a 
 function $\psi_1$  which tames the foliation $\FF_2$ on $S$ which is equal to $\FF$ on $D_2$ and to the foliation with 1 elliptic singularity on $D_2'$. We can assume that $\psi_1|_{\p D_1}= \psi_2|_{\p D_2}=1$. Concatenating functions $\psi_1|_{D_1}, \psi_2|_{D_2}$ and $\phi|_{A}$ we get the required taming function for $\FF$.
 \end{proof}
 \subsection{The parametric case}\label{sec:param}
  \begin{lemma}\label{lm:homotopy-taming}
 The space of  taming functions  for a tight foliation is  a convex set  $C^\infty(S)$.
 \end{lemma}
 \begin{proof}
  Convex combination of Lyapunov functions for the same characteristic foliation  $\FF$ is again a Lyapunov  function.   Note that according to Remark \ref{rem:taming-no-hom} we can assume that the tight foliation  $\FF$ has no homoclinic conditions. Hence, we can apply   the simplicity condition \eqref{eq:simplicity} which  survives taking  a convex combination.
     \end{proof}
     
Let us denote by  $\TT$ the space of  tight  generalized Morse foliations on the sphere $S=S^2$ and by  $\PP$ the space of pairs $(\phi,\FF)$, where $\FF\in\TT$ and $\phi$ is  its taming function.  
\begin{prop}\label{prop:taming-parametric}
Let $\Lambda$ be a compact manifold with boundary. Then for any maps  $h:\Lambda\to\TT$ and $H\colon\Op\p\Lambda\to\PP$ such that $\pi\circ H=h|_{\Op \p\Lambda}$ there exists a map $\wh H:\Lambda\to\PP$ such that $\wh H|_{\Op\p\Lambda}=H$ and $\pi\circ \wh H= h$.
\end{prop}
\begin{proof}
We will deduce the statement by proving a slightly weaker version of the claim that $\pi:\PP\to\TT$ is a Serre fibration with a contractible fiber.  The fiber is indeed contractible, as Lemma \ref{lm:homotopy-taming} shows. Let us  show the following  modified Serre covering homotopy property: {\em  given a map $f:D^k\times [0,1]\to \TT$ and a covering map $F:D^k\times[0,\eps)\to\PP, \;\pi\circ F=f|_{D^k\times[0,\eps)}$ there exists a lift
$\wh F: D^k\times [0,1]\to \PP$ such that $\pi\circ \wh F=f$ and $\wh F=F$  on $ D^k\times 0$.}
 We will need the following 
 \begin{lemma}\label{lm:lift-local}
  Given any map $h:D^k(1)\to \FF$    there exists $\eps\in(0,1]$ and a map $  H:D^k(\eps)\to X$ such that $p\circ H=h$.
  \end{lemma}
    The notation $D^k(r)$  stands here  for the centered at $0$ disk of radius $r$ in $R^k$. 
\begin{proof}
Denote  $\FF_s:=h(s)$, $s\in D^k(r)$.
According to Proposition  \ref{prop:tight-2-sphere} it admits  a taming function $\phi$.  

\noindent(i)   Suppose first that $\FF_0$ has no embryos. Then  for a sufficiently small $\eps>0$ and $|s|\leq \eps$ there exists a family of diffeomorphisms $\alpha_{s }$,  $\alpha_{ 0}=\Id$,   such that the foliation  $( \alpha_{s })_*\FF_{s}$ has the same critical points   as $\FF_{0}$, and  hence according to Lemma \ref{lm:taming-nearby} the functions $\phi_{s}\circ\alpha_{s} $ tame $\FF_{s}$  for $s\in D^k(\eps)$ if  $\eps$ is chosen small enough. Hence, we can set $H(s)=(\FF_s,\phi_s)$ for $|s|\leq \eps$.
 
\smallskip  \noindent (ii) If $\FF_0$ has embryo points then,    according to Lemma \ref{lm:local-taming}ii), there exists $\eps>0$ and a family of   functions $\psi_{s}$, $ |s|<\eps$,  on a neighborhood $U$ of embryos of $\FF_0$ such that $\psi_0=\phi|_U$, and $\psi_s$ is Lyapunov for $\FF_s$.     
   Take a smaller  neighborhood $U'\Subset U$ of the embryo locus and consider a  cut-off function $\theta :U\to[0,1] $ which is supported in $U$ and  equals 1 on  $U'$,  
     and define 
 $\phi_{s}|_U:=\theta\psi_s+(1-\theta)\phi.$  Then for a small enough $\eps$ and $|s|<\eps$ the function $\phi_s$  is Lyapunov for $\FF_s|_U$, and hence,   applying the
   argument in (i) to $\FF_s|_{S\setminus U}$   we can  extend  the family $\phi_s$ to $S$.

\end{proof}

Now we can conclude the proof of the modified Serre covering homotopy property.
  Lemma \ref{lm:lift-local} allows us to construct   a  finite covering $$\bigcup\limits_1^N U_j\supset D^k\times[ \eps/2,1],\;\;  U_j\subset D^k\times( \eps/ 3, 1],$$ and maps $F_j:U_j\to\PP$ covering $f|_{U_j}$. Consider a partition of unity $\theta_j$, $j=0,\dots, N$, subordinated to the covering $U_0:=D^k\times(0,\eps), U_1,\dots, U_N$. Set $F_0:=F$.  Denote
$\FF_s=f(s), s\in D^k\times[0,1]$, $F_j(s)=(\FF_s,\phi_s^j),\; j=0,\dots, N$ for $s\in U_j$,  we can define
   the required
lift $\wh F$ by the formula $\wh F(s)=(\FF_s,\sum\limits\theta_j\phi^j_s),\;  s\in D^k\times[0,1]$.
\end{proof}
     \section{Igusa type theorem}
  We prove in this section an analog of Igusa's theorem about functions with moderate singularities for normalization of singularities of characteristic foliations on a family of embedded surfaces. We follow the general scheme of our papers \cite{EM00,EM12}, while  specializing the argument for the specific context of the current paper. 
  \medskip
  
 Given a family $\xi_\lambda$ of  contact structures, or more generally  codimension $1$ distributions on a manifold $U$, we will view  it   as a fiberwise distribution $\xi=\{\xi_\lambda\}_{\lambda\in\Lambda}$ on the trivial fibration $\Lambda\times U\to\Lambda$.  The notation $\zeta_0$  stands for the standard contact structure on $\R^3$. The parameter space $\Lambda$  will be assumed to be  a    compact manifold with boundary $\p\Lambda$. 
  
 \begin{prop}\label{prop:Igusa}
 Let $\xi=\{\xi_\lambda\}_{\lambda\in\Lambda}$   be a  fiberwise   contact structure  on $\Lambda\times U$, where $U=\Op_{\R^3}D$ is a neighborhood of  the  2-disc  $D\subset \R^2=\{x_3=0\}\subset\R^3$. Suppose that   $\xi_\lambda $ is transverse to $D$  for $\lambda$ in     $\p\Lambda $. Then there exists a $C^0$-small fiberwise isotopy $\phi^t:D\times\Lambda\to\Op D\times\Lambda$, $t\in[0,1] $, such that 
 \begin{itemize}
 \item[a)]    $\phi^t =\Id$ on   $\Op( \p \Lambda\times D \cup \Lambda\times \p D ) $, $t\in[0,1]$;
   and
 \item[b)] $\phi^1(\Lambda \times D)$ has fiberwise  generalized Morse tangencies to $\xi$.
 \end{itemize}
 \end{prop}
  Denote $W:=\Lambda\times D$, $\wh W:=\Lambda\times U$, where $U$ is a neighborhood of $D$ in $\R^3$. Denote by $\Vert^3$ the rank $3$ vector bundle over $\wh W$ tangent to the fibers of the projection $\pi:\wh W\to \Lambda$. and by $\Vert^2$ the rank $2$ bundle over $W$ tangent to the fibers $W\to\Lambda$.
   We have $\p W=\p\Lambda\times D\cup\Lambda\times\p D$.
 
 We begin  the proof with two lemmas.
 A  codimension 2 submanifold  $V\subset W$  is said to be  in an {\em  admissible  folded  position} if the restriction $\pi|_V:V\to\Lambda$ has only fold singularities, the  fold $\Sigma$ divides $V$ into $V=V_+\cup V_-$,  $ V_+\cap V_- =\Sigma$, and the line bundle $\Ker(d\pi|_{TV}) $ over $\Sigma$ extends to $V_-$ as a trivial subbundle $\mu$ of $\Vert^3|_{V_-}$. We say that its fold $\Sigma=\Sigma^1( \pi|_V)$ is {\em small} if it is contained in a union of balls  in $  V$.

 \begin{lemma}\label{lm:Igusa1}
 Under an assumption of Proposition \ref{prop:Igusa} there exists a fixed on  $\Op( \p W) $  fiberwise homotopy $\xi^t$, $t\in[0,1]$, such that  the tangency locus $V=\{( \lambda,x)\in \Lambda\times D,  \xi^1_\lambda(x)=T_xD\}$ is a codimension 2 submanifold of $W$ which is in an admissible  folded position with a small fold.  \end{lemma}
 
 \begin{proof}   We  trivialize $\Vert^3$ by an orthonormal frame $e_1,e_2,e_3$ such that $e_1,e_2\in TD$, i.e. $e_1,e_2$ generate $\Vert^2$. Denote by $\Uert $ the associated to $\Vert^3$ unit  sphere bundle. The trivialization of $\Vert^3$ yields a splitting $\Uert=\wh W\times S^2$.  A fiberwise orthogonal to $\xi$   unit vector field $\nu$   yields a map $F:=F_\xi:W\to S^2$, and    
  $V:=  F^{-1}( e_3)\cup F^{-1}( -e_3) \subset W  $  is the tangency locus of $D$ to $\xi=\{\xi_\lambda\}$.   Perturbing, if necessary, $\xi$   we can assume that $\pm e_3$ are regular values of $F$, and hence $V$ 
   is a framed codimension 2 submanifold  of $\wh W$.
   
   The triviality of the normal bundle to $ V$ in $ W$ implies, see  e.g. Theorem 1.2 in  \cite{EM00},    that there exists a $C^0$-small diffeotopy    $\alpha_t: W \to W$, which is fixed over $\p\Lambda$, and  such that   $\alpha_1(V)$
   is in an admissible  folded position with a small  fold.
   Therefore, the homotopy of  plane distributions orthogonal to the  homotopy  of vector fields   $\nu\circ\alpha_t^{-1}\in\Vert^3$ has the required properties.
  \end{proof}
 
\begin{lemma}\label{lm:Igusa-local} \begin{enumerate}
\item Suppose that $ V={}^+V\cup {}^-V \subset W$  is in an    admissible  folded  position with small folds. Then there exists   a fiberwise contact structure  $\zeta $ on $\Op V\subset\wh W$ such that $F_{\zeta}^{-1}(\pm e_3)={}^\pm V$. 
 \item Given a  framing $F$ of $ V$ one can find a $C^0$-small isotopy $\alpha_t:V\to W$ and a fiberwise contact structure $\wt\zeta$ on $\Op\alpha_1(V)$   such that
\begin{itemize}
\item[-]  the isotopy $\alpha_t:V\to W$ is  supported away  from $\Sigma$; 
\item[-] $\alpha_1(V)$  is  in an admissible folded position, and
\item[-]  the push-forward framing $(\alpha_1)_*F$ is homotopic to the framing $F_{\wt\zeta}$.
\end{itemize}
\end{enumerate}
\end{lemma}
\begin{proof}
(i) Denote ${}^\pm\Sigma:={}^\pm V\cap\Sigma$. Choose coordinates  $(x,y,z,s,u)$  on a tubular neighborhood ${}^\pm N$ of ${}^\pm \Sigma$ in $ \wh W$ such that  the projection $\pi$ is given by $(x,y,z,s,u)\mapsto \lambda=(s,u)$, the coordinate $z$ is equal to $x_3$ (so that $W=\{z=0\}$)
 $V=\{s=y^2\}\subset W$, $\Sigma=\{s=0\}\subset V$, and  the vector field $\frac{\p}{\p y}$ along $\Sigma$ is inward transverse to the boundary of $V_+={}^-V_+\cup {}^+V_+$.   

Define the fiberwise  contact structure $\wh \zeta$ on ${}^\pm N$ by the $1$-form 
$\wh \alpha:=\pm dz+xdy+( s-y^2)dx$.   Let  us  change the variable $\wt y=s-y^2$  on  ${}^\pm N\cap \Op {}^\pm V_-$,   so that we have       ${}^\pm V_-=\{  x=\wt y=0\}$.
 
Note that   on ${}^\pm N\cap \Op {}^\pm V_-$ we have $\wh \alpha:=\pm dz+\wt ydx+\frac{x}{2\sqrt{s-\wt y}}d\wt y$. Choosing a smaller neighborhood ${}^\pm \wh N\Subset {}^\pm N$ we can find   a contact  form $\alpha$  on  ${}^\pm \wh N\cup( {}^\pm N\cap \Op {}^\pm V_-)$  which is equal to $\wh\alpha$ on ${}^\pm \wh N$ and equal to
$ \pm dz+  2x d\wt y +\wt ydx$ near $(\p {}^\pm N)\cap \Op ({}^\pm V_-)$.  Extend the coordinates $(x,\wt y)$ to $\Op({}^\pm V_-)$ in such a way that $\frac{\p}{\p\wt y}$ spans $\mu$  over $\Sigma$, and extend $ \alpha$ to 
$\Op({}^\pm V_-)$  as $ \pm dz+   2x d\wt y +\wt ydx$ to $\Op({}^\pm V_-)\setminus{}^\pm N$.
 
 Similarly, choosing  a  coordinate   $\wt y=y^2-s$ on   ${}^\pm N\cap \Op( {}^\pm V_+)$  we get  $\wh \alpha:=\pm dz-\wt ydx+\frac{x}{2\sqrt{ \wt y+s}}d\wt y$. Hence, we can find a contact form $\alpha={}^\pm \wh N\cup( {}^\pm N\cap \Op ({}^\pm V_+)$  which is equal to $\wh\alpha$ on ${}^\pm \wh N$ and equal to
$ \pm dz+   x d\wt y -\wt ydx=\pm dz+r^2d\phi$ near $(\p {}^\pm N)\cap \Op( {}^\pm V_+)$, where $r,\phi$ are the corresponding polar coordinates.
  Finally  we   extend $\alpha$  to  the rest of $\Op V_+$ as equal to  $\ dz+r^2d\phi$.

\medskip (ii) The required  framing $F$ differs from   the framing $\wt F$    defined  by the constructed structure  $\zeta$ by a homotopy class of a map $h:V\to S^1$. Let $\Delta\subset V\setminus B$ be a hypersurface dual to the cohomology class $h^*\mu$, where $\mu$ is a fundamental class of $S^1$.  If $\Delta$ is contained in $V_{-}$ then by twisting the coordinate system along fibers of the normal bundle to $\Delta$ we can make the framing $\wt F$ homotopic to $F$. Otherwise, by  a $C^0$-small isotopy supported in $\Op \Delta$ we create an  additional  double fold bounding a new component  $\Delta\times(-\eps,\eps)$ of $V_-$, see Fig. \ref{fig:pleating}.
    \begin{figure}[h]
\centerline{\includegraphics[width=11cm]{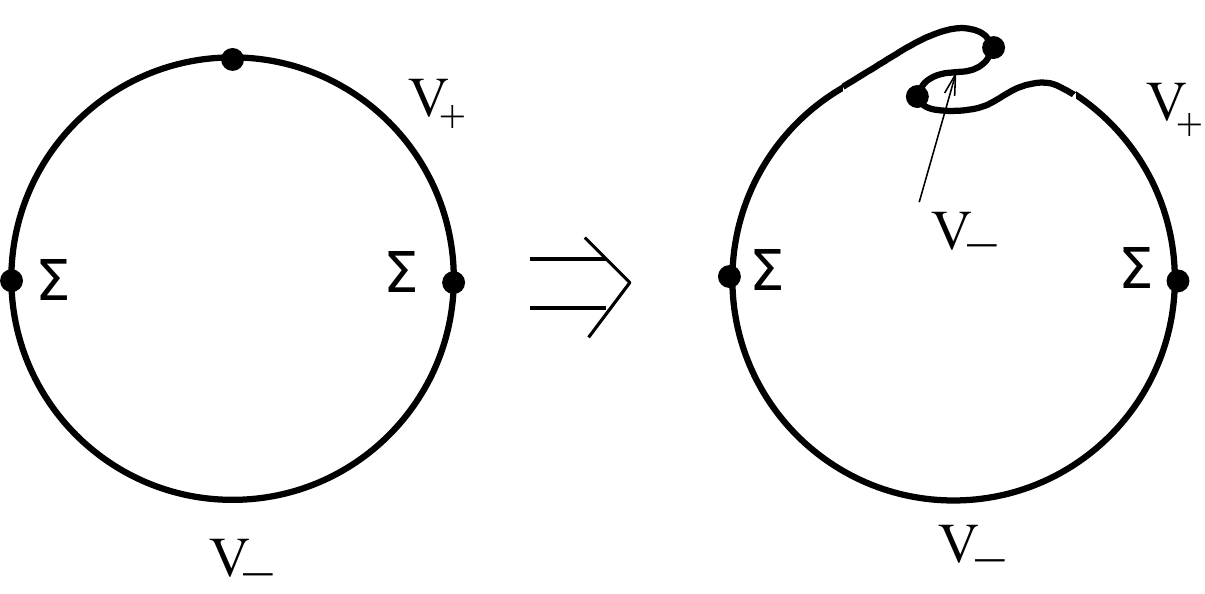}}
\caption{Creating a double fold along $\Delta$} 
\label{fig:pleating}
\end{figure}
 By appropriately twisting the line bundle $\lambda$ across this component, see Fig. \ref{fig:twisting}, we can change the homotopy class of the framing to make it coincide with $F$. 
 
    \begin{figure}[h]
\centerline{\includegraphics[width=9.5cm]{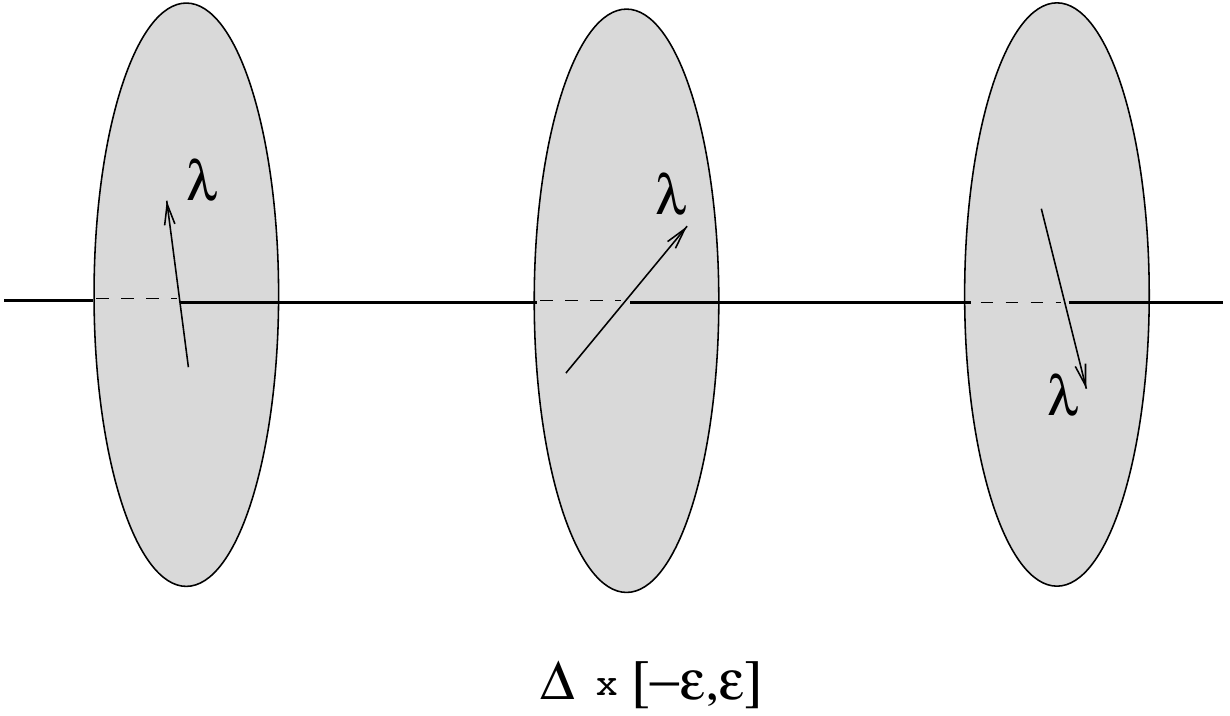}}
\caption{Twisting the line bundle $\lambda$ across $\Delta$} 
\label{fig:twisting}
\end{figure}
\end{proof}
 
\begin{proof}[Proof of Proposition \ref{prop:Igusa}]
Let $\wt\xi$ be a fiberwise distribution on $\wh W$ provided by Lemma \ref{lm:Igusa1}.
Using Lemma \ref{lm:Igusa-local} we can modify $\wt\xi$ to make it contact on $\Op V$ with the fiberwise  generalized Morse tangencies to   the fibers of the fibration $W\to\Lambda$.  There exists a fiberwise diffeotopy $\delta_t:\Op V\to \Op V$  such that  \begin{itemize}
\item[-] $\delta_t(V)=V$;
\item[-] $(\delta_1)_*\xi=\wt\xi$;
\item the homotopy $(\delta_t)_*\xi|_{V}$ extends  to  a homotopy of distributions $\wt\xi_t$ over $\wh W$   connecting $\xi$ and $\wt\xi$.
\end{itemize}

  Note that two fiberwise contact structures $(\delta_1)^{*}\wt\xi$ and  $ \xi|_{V}$ coincide along $V$, and hence, the fiberwise Darboux theorem implies that there exists    a fiberwise isotopy   $\alpha_t:\Op V\to \Op V$ such that $\alpha_t|_V=\Id$, $(\alpha_1)_*\xi=(\delta_1)^{*}\wt\xi$. Let $\beta_t$ be the concatenation of the isotopies $\alpha_t$ and $\delta_t$.  The homotopy of homomorphisms $d\beta_t:\Vert^3\to\Vert^3$ over $\Op V$ can be extended to $\Op W$ as a homotopy  $\Phi_t$ of injective homomorphisms. Let us consider the  fiberwise tangential homotopy $\tau_t=(\Phi_1\circ \Phi_t^{-1})_{TD}$. Using \cite{EM09} it can be integrated to a wrinkled isotopy $w_t$, $C^0$-close to $\tau_t$ and which coincide with $ \alpha_1\circ{\alpha_t}^{-1}$ on $\Op V$. Hence, $w_1$ is (after the smoothing) is the  required embedding. Note the the details of this construction in the case of non-integrable distributions is provided in the Honors  Bachelor dissertation of Ying Hong Tham, see \cite{Th16}.
\end{proof}
\section{Extension of contact structures to the $3$-ball}\label{sec:extension}

A characteristic foliation on the sphere $S$ is called {\em normalized} if it has  a fixed   positive and negative elliptic points at its poles, and a fixed arc $\gamma$ connecting the poles and transverse to $\FF$.

  A normalized characteristic foliation $(\FF,\gamma)$ is called {\em standard} if  it is diffeomorphic to the characteristic foliation on the   round  sphere $S_R=\{x_1^2+y_1^2+x_2^2=R^2\} \subset (\R^3, \{dz+r^2d\phi=0\})$ with a meridian connecting the poles as $gamma$.       The  space of standard   foliations is contractible,   and hence any standard foliation  has a homotopically canonical extension as a  tight contact structure on the ball $B$.

The goal of this section is  the following proposition. 
\begin{prop}\label{prop:ext-fol-cont0}
Let $\FF_\lambda$ be a family of tight normalized generalized Morse foliations on the sphere $S=\p B$, parameterized by a compact manifold $\Lambda$.  Suppose that $\FF_\lambda$ is standard for $\lambda\in\p \Lambda$. Then $\FF_\lambda$ extends to a family of tight contact structures $\xi_\lambda$  on the ball $B$, where the extension is standard for  $\lambda\in\p \Lambda$.
\end{prop}

\subsection{Special simple taming functions}\label{sec:Stein-type}
Let $\phi:S\to\R$ be a simple taming function for a  generalized Morse characteristic foliation $\FF$.  We call $\phi$ {\em special} if there exists a neighborhood $U\subset S$ of the critical point locus of $\FF$, a complex structure $i$ on $U$ and a smooth function $h:U\to\R$ such that  $\FF|_{\Op p}$ is given by a Pfaffian equation
\begin{equation}\label{eq:special}
d^\C\phi+hd\phi=0,\;\hbox{where}\; d^\C:=d\circ i.\end{equation} 
We will refer to $(i,h)$ as an {\em enrichment} of the special simple taming function $\phi$. 
We will always assume  the   complex  structure $i$  compatible with the orientation of $S$ near positive points and    opposite to it  near the negative ones. This assumption automatically makes   the function  $\phi$   subharmonic on a sufficiently small neighborhood of singular points.

   Given a triple $(\phi,i, h)$   we denote by $\FF( \phi,i,h)$ the   foliation  defined by  the Pfaffian equation \eqref{eq:special}. 
\begin{lemma}\label{lm:ct-Lyap} The function $\phi$ serves as a Lyapunov function for the foliation $\FF ( \phi,i, h)$.
\end{lemma}

\begin{proof}
In local holomorphic coordinate $z=x+iy$ we have $d^\C\phi=-\phi_ydx+\phi_xdy$, 
$d^C\phi=(-\phi_y+h\phi_x)dx+(\phi_x+h \phi_y)dy$, and hence, $\FF(\phi,i,h)$ is directed by the vector field $Z=(\phi_x+h\phi_y)\frac{\p}{\p x}+(\phi_y-h\phi_x)\frac{\p}{\p y}.$
Hence, we  have $$d\phi(Z)=\phi_x^2+\phi_y^2=||d\phi||^2\geq\frac{||Z||^2}{(1+|f|^2||)}.   $$
\end{proof}
An isolated  singularity of a foliation $\FF$ is said to be of complex geometric type, if it has the form $\FF(\phi,i,h)$ for some $\phi,i,h$. 

\begin{Example}\label{ex:Stein}
 The following are examples of singularities   of complex geometric type:  
(i)      $\FF$ is directed by a linear vector field $Z(u)=Au$ on $\R^2$, where $A$ is non-degenerate;

 (ii) an embryo singularity.
 \end{Example}

\begin{proof}
 
(i) Suppose first that   that $A$ is diagonalizable in a real basis, and hence, we can assume it is  diagonal, $A=\left(\begin{matrix}
\lambda&0\\0&\mu\end{matrix}\right)$.  Then  the tangent  to $\FF$ line field is defined by the 1-form $d^\C\phi$ for $\phi=\frac12(\lambda x^2+\mu y^2)$  and  the standard complex structure. If $A$ is elliptic and has complex eigenvalues then  we can assume that
$A=  \left(\begin{matrix}
a &-1\\1&a\end{matrix}\right), a,b\neq 0.$ Then the $1$-form annihilating  $Z$ is proportional to $\alpha=-(x+ay)dx+(-y+ax)dy$.
Take $\phi=\frac a2(x^2+y^2)$ and choose  $h=-1$.
Then $d^\C\phi-fd\phi =\alpha$.  

 Finally, suppose  that $A$ is in a Jordan form   $A=\left(\begin{matrix}
 \mu&1\\0&\mu\end{matrix}\right).$
 One can check that it corresponds to the singularity $\FF(i,\phi,f)$, where
   $\phi=ax_1^2+bx_1y_1+cy_1^2$ with  $a=2\mu^3,\; b=2\mu^2,\; c=\mu(1+2\mu^2), $ and $h=\frac{1+4\mu^2}{2\mu}.$  
 
  (ii) As it was already stated in Section \ref{sec:char-fol} a result of F. Takens in \cite{Ta74}  implies that in    suitable coordinate system the directing vector field $Z$  can be written as
  $Z=x\frac{\p}{\p x}+y^2f(y)\frac{\p}{\p y}$, see Section \ref{sec:char-fol}, and hence, it corresponds  to 
  $\FF(\phi, i, 0)$ with   $\phi=\frac12x^2+\int
  \limits_0^yu^2f(u)du$.
\end{proof}
  \begin{remark}
  It follows from a smooth version of Poincar\'e-Dulac theorem, see \cite{Du04,IY91}, that a Liouville vector field with a non-degenerate  zero is orbitally equivalent to its linear part, provided the non-resonance  condition for its eigenvalues, i.e. that there are no non-negative integers $n_1,n_2$ with $n_1+n_2\geq 2$ such that $\lambda_j=n_1\lambda_1+n_2\lambda_2,\; j=1,2$. Hence,  most (and possibly all)    generalized Morse singularities   are    of complex geometric type. 
  \end{remark}
  The following lemma, whose proof is straightforward, clarifies the geometric meaning of singularities of complex geometric type.
  
\begin{lemma}\label{lm:com-geom-meaning}

Consider  a function $\phi:(\C,0)\to(\R,0)$ with an isolated critical point at $0$.
In $\C^2$ with coordinates $(z_1=x_1+iy_1, z_2=x_2+iy_2)$ denote $$\R^3:=\{y_2=0\}, \C:=\{x_2=0\}\subset\R^3,
\Gamma_\phi:=\{x_2=\phi(x_1,y_1)\}\subset\R^3.$$ Consider  a    field of hyperplanes $\eta=\{dx_2-d\phi-hdy_2=0\}\subset T(\C^2)$ along the graph $\Gamma_\phi$.  Then the line field $i\eta(u)\cap T_u\Gamma_\phi\subset T_u\Gamma\phi,\;\; u\in\Gamma_\phi$,    
generates   the foliation $\FF(\phi,i, h)$.
\end{lemma}

We observe that the field $\eta$ can be  always  realized as a field of planes tangent  to   a hypersurface $\Sigma\supset\Gamma_\phi$ along $\Gamma_\phi$. The hypersurface $\Sigma$ is automatically   transverse to $\R^3$.
 The hypersurface $\Sigma$ is strictly pseudoconvex in a neighborhood of $0$ if and only if the function $\phi:\C\to\R$ is subharmonic, and  in that case the  foliation $\FF( \phi,i,h)$  coincides with the characteristic foliation induced by the contact structure on $\Sigma$ formed by  complex tangencies $T\Sigma\cap iT\Sigma$.
   
 Let $\FF$ be a characteristic foliation with singularities of complex geometric type and $\phi$ its special simple taming function with its enrichment $(j,h)$.  An immersion $f:S\to \R^3=\C\times\R$ is called {\em compatible} with $(\FF,\phi, j,h)$ if 
 $f$ has the form $g\times\phi$, where  $g:S\to\C$ is a smooth map, and the complex structure $j$ on a neighborhood of  a singular point $p$ of $\FF$  coincides with $g^*i$  if $p$ is  positive, and with $-g^*i$ if $p$ is negative. 
 
 \begin{lemma}\label{lm:global-Steiness}
  Let  $\FF$ be a characteristic foliation on $S$, $\phi$ its special taming function with an enrichment $(j,h)$, and $f:S\to\R^3=\C\times\R$ a compatible  immersion. Let us identify $\R^3$ with the  subspace $\{y_2=0\}$ in $\C^2$ with coordinates $(x_1+iy_1, x_2+iy_2)$.
   Then $f$ extends to  an immersion $F:S\times(-\eps,\eps)\to\C^2$ such that 
\begin{itemize}
\item[(i)] $F$ is transverse to $\R^3$;
\item[(ii)] $\Sigma:=F(S\times(-\eps,\eps))$ is strictly pseudoconvex;
\item[(iii)] $\FF$ is the characteristic foliation on $S$ induced by the contact structure $\xi$ on $\Sigma$ defined by the field of complex tangencies $T\Sigma\cap iT\Sigma$.
\end{itemize}
\end{lemma}
\begin{proof} The construction is local, and hence, to simplify the notation we will be assuming that $f$ is an embedding. 
 
 We claim that the vector field $\frac{\p}{\p y_2}$ over $f(U)$ uniquely up to scaling extends to $f(S)$ as a transverse to $\R^3$ vector field   which satisfies the following property:
  \begin{itemize}
  \item[-]  for any regular  point  $u\in  S$ of $\FF$  we have $$d_uf^{-1}(i\Span(v(f(u)), d_uf(T_uS)) =T_u\FF.$$
  \end{itemize}
  Existence of such a vector field  $v$    is guaranteed by 
  the following lemma from Linear Algebra.
  
  \begin{lemma}\label{lm:lin-alg2}
    Consider $\C^2\supset \R^3=\{y_2=0\}\supset\C=\{x_2,y_2=0\}$. Let $L\subset\R^3$ be a 2-dimensional subspace transverse  to $\C$. Let $\ell$ be a line in $L$ transverse to $L\cap\C$. Then there is a unique real 3-dimensional subspace $P\subset\C^2$ which is transverse to $\R^3$ and  such that $P\supset L$ and $iP\cap L=\ell.$ 
    \end{lemma}   
    \begin{proof} It is sufficient to consider the case \begin{align*}
    &L=\{x_1=k x_2, y_2=0\},\\ & \ell=  \{y_1=mx_2\}\cap L.
    \end{align*}
    The equation for $P$ can be written in the form
    \begin{align*}
    &x_1-kx_2=cy_2,
    \end{align*}
    and hence
    \begin{align*}
    &iP:=y_1-ky_2=-cx_2.
    \end{align*}
    Then $$iP\cap L=\{y_1=-cx_2, x_1=kx_2,\},$$
    i.e. $c=-m$. 
    \end{proof}

  Continuing   the proof of Lemma \ref{lm:global-Steiness}, take  a normal vector field $\nu$ to   $f( S)$ in $\R^3$,  chosen to coincide  with $\frac{\p}{\p x_2}$ at positive singular points of $\FF$. Consider a map $\wh F:S\times\R^2\to\C^2$ given by the formula  $\wh F(u, t_1,t_2)=f(u)+t_1v(f(u))+t_2\nu(f(u))$, and define a map
  $F:S\times[-\eps,\eps]\to \C^2$ as
  $F(u,t)=\wh F(u, t, Ct^2)$. We claim that $F$ is the required immersion if $C>0$ is sufficiently large  and $\eps>0$ is sufficiently small.  Near the critical point locus of $\FF$ this follows from  Lemma \ref{lm:com-geom-meaning}   for any $C$, even negative. Away from the critical locus property  (iii) holds  by construction,  while (ii)   can be achieved by choosing a sufficiently large $C$, see \cite{CE12} for the details. 
   \end{proof}
 
 \subsection{Families of special  taming functions}
 Considering families of special taming functions we need their enrichments to  depend continuously on the parameter.
 More precisely, let  $\{\FF_\lambda,\phi_\lambda\}_{ \lambda\in\Lambda}$ be a family of characteristic foliations and their simple taming functions. Let us view   $\{\FF_\lambda\}$ as a fiberwise foliation on $W=\Lambda\times S$.  Let $V$ be its fiberwise singular locus. The family
 $\{\phi_\lambda\}$ is called {\em special}  if   there is a  fiberwise complex structure $I$ and  function $H$ on a neighborhood $U\supset V$ in $W$ such that,  $\FF_\lambda|_{U_\lambda}=\FF(\phi_\lambda|_{U_\lambda}, I|_{U_\lambda}, H|_{U_\lambda})$, where $U_\lambda=U\cap (\lambda\times S)$, $\lambda\in\Lambda$.
  
\begin{prop}\label{prop:imm-fol-cont}
 Let $\{\FF_\lambda,\phi_\lambda\}_{\lambda\in\Lambda}$ be a family of normalized generalized Morse foliations together with their special simple taming functions. 
   Then there exists a  family of immersions $f_\lambda:B\to\R^3=\C\times\R$  such that $ F_\lambda|_{S=\p B}$ is compatible with    $(\FF_\lambda,  \phi_\lambda)$.  
\end{prop}
\begin{proof}
According to Proposition \ref{prop:simple-extension} the family $\{\phi_\lambda\}$ extends to a family $\{\Phi_\lambda\}$ of simple functions on the ball $B$.
    Consider the map ${  \Phi}:\Lambda\times B\to\Lambda\times \R$, given by $ {  \Phi}(\lambda, u)=(\lambda,\Phi_\lambda(u)), \lambda\in\Lambda, u\in B$. Slightly enlarging $B$ to an  open ball $\wt B\Supset B$ and extending there the map ${  \Phi}$, consider the foliation $\HH$ of $\Lambda\times \wt B$ by the level sets ${   \Phi}^{-1}(\lambda,t), t\in\R,\lambda\in \Lambda$.  Recall that for  each critical point $p_\lambda$ of the function $\phi_\lambda $ there is  a neighborhood $U_\lambda\subset S$ with a complex structure $i_\lambda$.  Choose a collar $A=\p B\times[-\eps,\eps]\subset \wt B$. The projection $\pi:A\to\p B$   allows us to identify leaves of $\HH_\lambda$ near each singular point $p_\lambda$,  and also to   define a leafwise complex structure on  leaves of $\HH$ in a neighborhood of $p_\lambda$ by inducing it from $i_\lambda$. We will continue using the notation $i_\lambda$ for the induced local complex structure on leaves of $\HH$.
 The tangent bundle to the foliation  $\HH$ is trivial, thanks to the normalization condition.
 This allows us to apply  Hirsch-Smale parametric $h$-principle for immersions of open manifolds to  construct a leafwise immersion $G:\Lambda\times \wt B\to\C$. Moreover, we can arrange  that on  a neighborhood of each singular points $p_\lambda$ the  map is constant on fibers of the projection $\pi$. Recall that  $i_\lambda$  compatible with the orientation of $S$ near positive points and opposite to it near negative ones. Thanks to the simplicity  condition   for the function $\Phi_\lambda$ the complex orientation is compatible  with the orientation of $\HH$. This allows us to choose the leafwise immersion $G$  leafwise $(i_\lambda,i)$-holomorphic near positive points, and anti-holomorphic near negative ones.  Then the
map $F:=G\times{ \Phi}: \Lambda\times \wt B\to \C\times\R$ yields a family of immersions
$ f_\lambda:B\to\R^3$, $f_\lambda( u)=(G(\lambda, u),\Phi_\lambda(u))$, whose restrictions to the sphere $S$ are compatible with   $(\FF_\lambda,\phi_\lambda)$.
\end{proof}
\subsection{Extension of contact structures}\label{sec:extension2}

  \begin{prop}\label{prop:ext-fol-cont}
 Consider a family $ \{(\FF_\lambda,\phi_\lambda)\}_{\lambda\in\Lambda}$ of normalized generalized Morse foliations together with their special simple taming functions.  Let   $f_\lambda:B\to\R^3=\C\times\R$  be a   family of immersions    compatible with   $(\FF_\lambda,\phi_\lambda)$.
 Then there exists a family of contact structures $\xi_\lambda$ on $B$ which induce the characteristic foliation $\FF_\lambda$ on $S$. If $\FF_\lambda$ is standard then $\xi_\lambda$ is standard as well.
\end{prop}
 
\begin{proof}
 According to Lemma \ref{lm:global-Steiness} the family of  immersions $f_\lambda|_S$ extends to a family of  immersions
 $F_\lambda:S\times(-\eps,\eps)\to\C^2$ such that 
\begin{itemize}
\item[-] $F_\lambda$  is transverse to $\R^3$;
\item[-] $\Sigma_\lambda:=F_\lambda(S\times(-\eps,\eps))$ is strictly pseudoconvex;
\item[-] $\FF_\lambda$ is the characteristic foliation on $S$ induced by the contact structure $\xi_\lambda$ on $\Sigma_\lambda$ defined  by the field of complex tangencies.
\end{itemize}
Consider the family $C_s:=\{x_1^2+y_1^2+x_2^2+(y_2-\frac1s+\sqrt{s})^2\leq\frac1{s^2}, s>0\}$.
Fix  a sufficiently small $s$ such that the sphere $\p C_s$ intersects  transversely $\Sigma_\lambda$ along a closed submanifold for all $\lambda\in\Lambda$. There is a family of immersions $\ol F_\lambda:B\times[-\eps,\eps]\to\C^2, \lambda\in\Lambda,$ such that $\ol F|_{B\times 0}=f_\lambda$ and $\ol F_\lambda|_{S\times[-\eps,\eps]}=F_\lambda$.
Let $j_\lambda$  be  the induced complex structure $\ol F_\lambda^*i$ on  $B\times[-\eps,\eps]$. Set $T_{\lambda}:=\ol F_\lambda^{-1}(\p C_s)$. Smoothing the corner  along $\p T_\lambda$ of the piecewise smooth  $3$-ball $T_\lambda\cup (S\times [0,\eps))$ we get a strictly pseudoconvex ball $\wh B_\lambda$   bounded by $S$. The corresponding contact structure $\xi_\lambda$ defined by complex tangencies is the required extension of the characteristic foliation $\FF_\lambda$. 

Note that  if $\FF_\lambda$ is standard then the  above constructed contact extension $\xi_\lambda$ is standard as well.
\end{proof}
\begin{remark}\label{rem:ext-tight} {\em The contact extension $\xi_\lambda$ provided by the above proposition is tight.} Indeed, we  could similarly construct a  strictly pseudoconvex  $\Sigma_\lambda\supset \wh B_\lambda$ which bounds a complex $4$-ball. Hence, the  contact structure induced by the field of complex tangencies on $\Sigma$ is holomorphically fillable, and therefore, tight.

While this observation is not needed for the proof  of Proposition  \ref{prop:ext-fol-cont0}, and hence, for the main result of this paper, we will use in the proof of Corollary \ref{cor:taming-tight}.
\end{remark}

  \begin{prop}\label{prop:constr-enriched-taming}
 Let $\{\xi_\lambda\}_{\lambda\in\Lambda}$ be a family  of contact structures on the spherical annulus $A:=S\times[-\eps,\eps]$, and $  \FF_\lambda,\; \lambda\in\Lambda,$ a characteristic foliation induced by $\xi_\lambda$ on the sphere $S=S\times 0$. Let    
$ \{\phi_\lambda\}_{\lambda\in\Lambda}$  be a  family of simple taming functions for  $\{\FF_\lambda\}_{\lambda\in\Lambda}$. Denote $\wh W:=\Lambda\times A$, $W:=\Lambda\times S=\Lambda\times (S\times 0)\subset\wh W$, and let $V\subset W$ will be the fiberwise singular locus of the fiberwise foliation ${\bf F}:= \{\FF_\lambda\}$.
Then there exists a fiberwise isotopy $J_{s,t}:=\{j_{s,t,\lambda}\}_{\lambda\in\Lambda}:W\to\wh W$, $s\in[-1,1], t\in[0,1]$, supported  in arbitrary small neighborhood $\Omega\supset V$, and  a family of functions $\Phi_{s}:W\to\R$, $ s\in[-1,1]  $, with the following properties
\begin{itemize}
\item[-] $J_{s,0}$   is the inclusion $W\hookrightarrow \wh W$ for all $  s\in[-1,1]$;
\item[-] $J_{ 1,t}(W)\subset \wh W_+:=\Lambda\times(S\times [0,\eps])$, $J_{-1,t}(W)\subset \wh W_-:=\Lambda\times( S\times [-\eps,0])$ for $ t\in[0,1]$;
 \item[-]  $ \phi_{\lambda,s}:=\Phi_s|_{\lambda\times S}$   is a simple taming function for  
 the characteristic foliations $$  \FF_{\lambda,s}:=j_{\lambda,s,1}^*\xi_{\lambda}, \;\lambda\in\Lambda, s\in[-1,1];$$ 
\item [-]    the characteristic foliations  $\FF_{\lambda,s}$ and $\FF_{\lambda}$ have the same singular locus $V_\lambda:=V\cap{\lambda\times S}$ for any $\lambda\in \Lambda, s\in[-1,1] $;
   \item[-] there exist
 a  fiberwise   complex structure   $I:=\{i_\lambda\}$ on $\Op V $, $\lambda\in\Lambda$, and       functions   $H:=\{h_\lambda\},\wh \Phi:=\{\wh\phi_\lambda\}:\Op V \to\R$,    
     such  that $\Phi_{  s}=\wh\Phi$ on $\Op V $,   
   $\Phi_s=\Phi $ on $W\setminus\Omega$, and
  $\FF_{\lambda,s}|_{\Omega}=\FF(\wh\phi_\lambda, i_\lambda,h_\lambda)$,  
 $ \lambda\in\Lambda, s\in[-1,1].$ In other words,  $\{\wh\phi_\lambda\}_{\lambda\in\Lambda}$ serves as a family of special simple taming functions for  the family  $\{\FF_{\lambda, s}\}_{\lambda\in\Lambda}$ for any $s\in[-1,1]$.
  
\end{itemize}

\end{prop}
 \begin{proof}
 Denote by  $\Sigma\subset  V$ the set of fiberwise embryo points.
According to Proposition \ref{prop:IY-param}  there are fiberwise local coordinates $(x ,y)$ on a  sufficiently small neighborhood $U\supset\Sigma$ in $W$ such that the  fiberwise foliation ${\bf F}|_U$ is generated by a fiberwise Liouville field  
${\bf Z}=x\frac{\p}{\p x}+F(y,\lambda)\frac{\p}{\p y}.$ Hence, $\FF_\lambda|_{\Op\Sigma\cap (\lambda\times S)}=\FF(\wh\phi_\lambda,i_\lambda, 0)$, where $i_\lambda$ is the fiberwise complex structure given by the fiberwise complex coordinate $x+iy$, and  
 $$\wh\phi_\lambda(x,y):=\frac{x^2}2+\int\limits_0^yf(u,\lambda)du.$$

Choose a  small neighborhood $U_1\Subset U, U_1\subset\Sigma $, and 
consider a cut-off function $\sigma: U\to[0,1]$ supported in   $U$  and equal to $1$ on $U_1$.
For each point $w=(0,y_0,\lambda)\in (V\setminus\Sigma) \cap U$ we have ${\bf Z}=x\frac{\p}{\p x}+ \frac{\p f}{\p y}(y_0,\lambda)(y-y_0)+h(y, \lambda))$, where $h(y, \lambda)=o(y-y_0)$. We define the new fiberwise vector field $\wh{\bf Z}=\{\wh Z_\lambda\}_{\lambda\in\Lambda}$ on $U$ by the formula
$$\wh Z_\lambda: =x\frac{\p}{\p x}+ \left(\sigma\frac{\p f}{\p y}(y_0,\lambda)+1-\sigma\right)(y-y_0)+\sigma h(y,\lambda).$$ 
The complement  $V\setminus U_1$  can be written as $E \cup H $, where $E $  is the locus of   elliptic,  and $H $ of   hyperbolic points.  We note that that the fiberwise  coordinate system $(x,y)$ (re-centered  to the points of $H$) can be extended to $\Op H$ in such a way that  the vector fields $\frac{\p}{\p x}$, $\frac{\p}{\p x}$ along $H$ are eigenvectors of the fiberwise linearization
of ${\bf Z}$.  Hence, we can extend $\wh{\bf Z}$ to $\Op H$ as $x\frac{\p}{\p x}-y\frac{\p}{\p y}  $.
Choosing  a fiberwise metric on  $\Op E$ which extends the metric $dx^2+dy^2$ from $U$   we  introduce fiberwise polar  coordinates  $(r,\theta)$  in the tubular coordinates of  $E$, so that the vector field $\wh{\bf Z}$ on
 $\Op\p U\cap\Op E$ is equal to $r\frac{\p}{\p r}$. Hence it can be extended as  $r\frac{\p}{\p r}$ to the rest of $\Op E$. Let  $\wh\FF_\lambda$ be the foliation generated   $\wh Z_\lambda$ on $U_\lambda$. The fiberwise  potential $\wh\phi_\lambda $ and its enrichment    $(i_\lambda, h_\lambda)$    extend to $\Op V$  in an obvious way, see  Example \ref{ex:Stein},  to satisfy
 $\wh\FF_\lambda =\FF(\wh\phi_\lambda,i_\lambda, h_\lambda)$  on $U_\lambda$.  
  
 It remains to  apply Lemma \ref{lm:local-type} to construct a 2-parametric fiberwise isotopy $J_{s,t} =\{j_{\lambda,s,t}\}_{\lambda\in\Lambda} :W\to\wh W, s\in[-1,1]$, which is supported in $\Op V$ away from $\Op\Sigma$,  and which satisfies the following properties
\begin{itemize}
\item[-] $J_{s,0}$ is the inclusion $V\hookrightarrow W$;
\item[-] $J_{-1,t}(W)\subset \wh W_-:=\Lambda\times S\times[-\eps,0]$, \;$J_{1,t}(W)\subset \wh W_+:=\Lambda\times S\times[0,\eps]$.
\item[-] the induced  characteristic foliation 
$\FF_{\lambda,s}:=j_{\lambda,s,1}^*\xi_\lambda$, $\lambda\in\Lambda, s\in[-1,1]$, coincides  with $ \wh\FF_\lambda$ on    $\Op V\cap(\lambda\times S)$, has $V\cap(\lambda\times S)$ as its  singular locus, and admits a  simple   taming function $  \phi_{\lambda,s}, \lambda\in\Lambda, s\in[-1,1],$  which coincides with $\wh\phi_\lambda $ near  its singular locus, and with $ \phi_{\lambda} $ outside a larger neighborhood of $V$.
\end{itemize}
Thus,  $   \phi_{\lambda,s}$ is a special simple taming function for   $  \FF_{\lambda,s} $  with an enrichment $( i_\lambda, h_\lambda)  $f
or each $s\in[-1,1]$.
 \end{proof}

Now we are ready to prove the main proposition of this section.

\begin{proof}[Proof of Proposition \ref{prop:ext-fol-cont0}]
 Using Proposition \ref{prop:taming-parametric} we can construct a family of simple taming functions $\phi_{\lambda}$, $\lambda\in\Lambda,  $ for $ \FF_{\lambda}$.
  
 Choose an extension of the family $\FF_\lambda$  as a family of tight  contact structures $\zeta_\lambda$ on  a collar $U:=S\times[-\eps,\eps]$.  Using   Proposition \ref{prop:constr-enriched-taming} we can construct a family of embeddings $ j_{\lambda, s}:S\to U,   $, and a family of functions $\phi_{\lambda,s},  \lambda\in\Lambda, s\in[-1,1],$ such that
\begin{itemize} 
\item[-] the characteristic foliation $\FF_{\lambda,s}:=j_{\lambda, s}^*\zeta_\lambda$ is normalized and $\phi_{\lambda,s}$ serves as     its  special simple taming function; 
\item[-] $j_{\lambda,1}(S)\subset S\times[0,\eps]$; $j_{\lambda,-1}(S)\subset S\times[-\eps, 0]$.
\end{itemize} 
Let $\wt B:=B\cup U  $, be  a larger ball bounded by  $S\times \eps$. Let us extend $j_{\lambda,s}$ to a family of embeddings $\wt j_{\lambda,s}:B\to\wt B$. Denote $B_\lambda:=\wt j_{\lambda,1}(B)\subset \wt B$.
 Using  Proposition \ref{prop:simple-extension} let us extend $\phi_{\lambda,s}$ to a family of simple functions $\Phi_{\lambda,s}$ to $B$, and then using Proposition \ref{prop:imm-fol-cont}  construct a family of compatible with $(  \FF_{\lambda,s},\Phi_{\lambda,s})$  of immersions $f_{\lambda,s}:B\to\R^3$.  Applying Proposition \ref{prop:ext-fol-cont} we extend the foliations $  \FF_{\lambda,s}$ as tight contact structures $\xi_{\lambda,s}$ to $B$.  
 
 For each $\lambda\in\Lambda$ and $s\in[-1,1]$ let us consider a contact structure $\eta_{\lambda,s}$ on $B_\lambda$   which is equal to $(J_{\lambda,s})_*\xi_{\lambda,s}$ on  $J_{\lambda,s}(B)\subset \wt B $ and equal to $\zeta_\lambda$ elsewhere. Note that $\eta_{\lambda,-1}$ induces the characteristic foliation $\FF_\lambda$ on $S=\p B$.  The contact structure $\eta_{\lambda,1}=(j_{\lambda,1})_*\xi_{\lambda,1}$  is tight, and hence, by Gray's stability  all contact structures $\eta_{\lambda,s}$ on $B_\lambda$  are tight. Therefore,
 $\eta_{\lambda,-1}|_B$ is the required tight extension of the characteristic foliation $\FF_\lambda$.
  
\end{proof}
\begin{prop}\label{cor:taming-tight} Any generalized Morse foliation $\FF$  on the sphere $S$ which admits a simple taming function is tight.
\end{prop}
\begin{proof} If $\FF$ admits a {\em special} simple taming function then the claim follows from  Proposition \ref{prop:ext-fol-cont} and Remark \ref{rem:ext-tight}. In the general case, we extend $\FF$ to a contact structure $\xi$  on an annulus  $S\times[-\eps,\eps]$ and  find an isotopy $j_s:S\to S\times[-\eps,\eps], s\in[-1,1]$, such that $j_1(S)\subset S\times[0,\eps)$, $j_{-1}(S)\subset S\times[-\eps,0]$ and the family of foliations $\{\FF_s:=j_s^*\xi\}$ admits a family of special taming functions.
  Arguing as in the proof of Proposition  \ref{prop:ext-fol-cont0}  we conclude  that $\FF$ embeds into a tight contact ball bounded by $\FF_1$, and therefore it is tight itself.  \end{proof}
\begin{remark} Note that the simplicity assumption for a taming function $\phi$ was used in the proof twice. First, in  the proof of Lemma \ref{lm:simple-restrict}(ii) in order to  extend the taming function $\phi$ to the ball, and second time in the proof of Proposition \ref{prop:imm-fol-cont} to construct a compatible immersion to $\R^3$.
\end{remark}
  \section{Proof of   the main theorem} \label{sec:proof-main}
   Let us choose the contact form $dz+r^2d\phi$ on $\R^3$ for the standard contact structure $\zeta_0$.
 Let $\xi_\lambda$, $\lambda\in \Lambda$, be a family of tight contact structures  on $\R^3 $  which coincides with   $\zeta_0 $ outside of a  compact set $K$. We can assume that $K$ is a ball of radius $1$ centered at a point with cylindrical coordinates $z=0,\phi=0, r=3$. Consider the family $B_r$ of balls of radius $r$ centered at $0$, $r\in[1,5]$, so that $B_5\supset K$ and $B_1\subset \R^3\setminus K$. Choose a family of meridians $\gamma_r\subset \p B_r$ connecting the poles. 
  Note that  the characteristic foliations $(\FF_{\lambda,r},\gamma_r)$ induced by $\xi_\lambda$ on $ \p B_r$ are normalized.
 
    Applying  Proposition \ref{prop:Igusa} to the  complements of neighborhoods of $\gamma_r\subset \p B_r$, we  can arrange that   the   characteristic foliations  $\FF_{\lambda,r}$  are generalized Morse.    Hence, we can use Proposition \ref{prop:ext-fol-cont0} to find  extensions of $\FF_{\lambda,r}$ to $B_r$ as tight contact structures $\zeta_{\lambda,r}$  on $B^r$ $r\in [1,5]$. Moreover, for $r=1,5$ and all $\lambda$ and for $\lambda\in\Lambda_0$ and all $r\in[1,5]$ the foliations $\FF_{\lambda,r}$ are standard. Therefore, their extensions $\zeta_{\lambda,r}$  are standard as well.

  Denote by $\eta_{\lambda,r}$ the contact structure which is equal to $\xi_\lambda$ on $\R^3\setminus B_r$, and equal to $\zeta_{\lambda,r}$ on $B_r$.
 Then $\eta_{\lambda,1}=\xi_\lambda$ while $\eta_{\lambda,5}=\zeta_0$.
This concludes the proof of Theorem \ref{thm:main}.

 \end{document}